\documentclass[11pt,reqno]{amsart}

\usepackage{amsfonts, amssymb, amscd}
\usepackage{amsmath}
\usepackage{verbatim}
\usepackage{amssymb}
\usepackage{mathrsfs}
\usepackage{graphicx}
\usepackage{cite}
\usepackage[all]{xy}
\usepackage{tikz}
\usepackage[toc,page]{appendix}
\usepackage{tikz-cd}

\usepackage{graphicx}
\usepackage{float}
\usepackage{makecell}

\usepackage{mathtools}
\usepackage{microtype}
\usepackage{hyperref}

\DeclarePairedDelimiter\floor{\lfloor}{\rfloor}

\def\rk {{\operatorname{rk}}}
\def\C {{\mathbb C}}
\def\CC {{\mathbb C}}

\def\PP {{\mathbb P}}

\def\Pf {{\mathrm {Pf}}}

\allowdisplaybreaks[3]

\theoremstyle{definition}

\newtheorem{theorem}{Theorem}[section]
\newtheorem{lemma}[theorem]{Lemma}
\newtheorem{proposition}[theorem]{Proposition}
\newtheorem{definition}[theorem]{Definition}
\newtheorem{example}[theorem]{Example}

\newtheorem{remark}[theorem]{Remark}
\newtheorem{conjecture}[theorem]{Conjecture}
\newtheorem{assumption}[theorem]{Assumption}
\numberwithin{equation}{section}
\begin{document}
\title[Stringy Hodge numbers of Pfaffian double mirrors and HPD]{Stringy Hodge numbers of Pfaffian double mirrors and Homological Projective Duality}

\begin{abstract}
We study the stringy Hodge numbers of Pfaffian double mirrors, generalizing previous results of Borisov and Libgober. In the even-dimensional cases, we introduce a modified version of stringy $E$-functions and obtain interesting relations between the modified stringy $E$-functions on the two sides. The modified stringy $E$-functions can also be interpreted as the stringy $E$-functions of certain ideal sheaf pairs. We use them to make numerical predictions on the Lefschetz decompositions of the categorical crepant resolutions of Pfaffian varieties.
\end{abstract}

\author{Zengrui Han}
\address{Department of Mathematics\\
University of Maryland\\
College Park, MD 20742} \email{zhan223@umd.edu}

\maketitle

\tableofcontents

\section{Introduction}\label{sec-intro}

The classical formulation of mirror symmetry concerns the duality of Hodge numbers of mirror varieties. More precisely, if $X$ and $X^{\vee}$ are $n$-dimensional mirror Calabi-Yau varieties, then one may expect
\begin{align*}
	h^{p,q}(X)=h^{n-p,q}(X^{\vee}),\quad\text{for all }p,q.
\end{align*}
However, the usual Hodge numbers are not well-behaved in the context of mirror symmetry: the expected equalities fail for singular mirror pairs.

\smallskip

This issue was addressed by Batyrev \cite{Batyrev}, who introduced the concept of \emph{stringy $E$-functions} as a replacement for the usual Hodge-Deligne polynomials for varieties with log-terminal singularities. The stringy $E$-function $E_{\mathrm{st}}(X;u,v)$ of a variety $X$ is a function in two variables defined as a certain weighted sum over strata of a log resolution of $X$. Importantly, the definition is independent of the choice of the log resolution.

\smallskip

Stringy $E$-functions behave well in the context of mirror symmetry: they agree with the usual Hodge-Deligne polynomials for smooth varieties, and more importantly, if a singular variety $X$ admits a crepant resolution $\widehat{X}$, then its stringy $E$-function $E_{\mathrm{st}}(X;u,v)$ is equal to the ordinary $E$-polynomial of $\widehat{X}$. This is consistent with the well-known principle in physics that string theories on singular spaces should be equivalent to string theories on their crepant resolutions.

\smallskip

In recent years, a phenomenon known as the \textit{double mirror phenomenon} has attracted great interest. A pair of spaces $X$ and $Y$ are said to form a double mirror pair if they share the same mirror family. Consequently, they are expected to have the same mirror-symmetry-theoretic properties (Hodge numbers, derived categories, elliptic genera, etc.).

\smallskip

Within the framework of the \textit{gauged linear sigma model} (GLSM) introduced by Witten \cite{Witten}, a double mirror pair can be seen as distinct phases of the same GLSM. The primary focus of this paper is the investigation of the \textit{Pfaffian double mirrors} construction, which serves as an example of a nonabelian GLSM. In contrast to the abelian GLSMs (toric mirror symmetry), where different geometric phases in the GLSM turn out to be birationally equivalent to each other (cf. \cite{Li}), this is not true for nonabelian GLSMs. 

\smallskip

The study of the Pfaffian double mirrors originates from the work of R\o dland \cite{rodland}. We briefly recall the original construction. Let $V$ be a 7-dimensional complex vector space, and $W$ be a generic 7-dimensional subspace of the space $\wedge^2 V^{\vee}$ of skew forms on $V$. We construct a pair of Calabi-Yau threefolds $X_W$ and $Y_W$ as follows:
\begin{itemize}
	\item $X_W$ is the subvariety of the Grassmannian $G(2,V)$ consisting of all 2-dimensional subspaces $T_2\subseteq V$ such that $w|_{T_2}=0$ for all $w\in W$.
	\item $Y_W$ is the intersection of the Pfaffian variety $\Pf(V)$ (i.e., the variety of all \textit{degenerate} skew forms on $V$) with the projective space $\PP W$ inside $\PP(\wedge^2 V^{\vee})$.
\end{itemize}
R\o dland argued that $X_W$ and $Y_W$ share the same mirror family. This example was further studied by Borisov-C\u{a}ld\u{a}raru \cite{BC}, Hori-Tong \cite{Hori-Tong}, and Kuznetsov \cite{K1}.

\subsection{Pfaffian double mirrors}

The generalization of the original construction of R\o dland to higher dimensions was suggested by Kuznetsov \cite{K1}. We briefly recall the construction here. See section \ref{sec-background} for further details.

\smallskip

Let $V$ be an $n$-dimensional complex vector space\footnote{Note that the construction differs for even and odd $n$.}. We consider the pair of dual Pfaffian varieties $\Pf(2k,V^{\vee})$ and $\Pf(2\floor*{\frac{n}{2}}-2k,V)$. We define a pair of Calabi-Yau complete intersections $X_W$ and $Y_W$ in these varieties associated to a generic subspace $W\subseteq \wedge^2 V^{\vee}$ of dimension $l$ as follows.

\begin{itemize}
	\item $X_W$ is the complete intersection of $\PP W^{\perp}$ and the Pfaffian variety $\Pf(2k,V^{\vee})$ in $\PP(\wedge^2 V)$.
	\item $Y_W$ is the complete intersection of $\PP W$ and the Pfaffian variety $\Pf(2\floor*{\frac{n}{2}}-2k,V)$ in $\PP(\wedge^2 V^{\vee})$.
\end{itemize}

\smallskip

Borisov and Libgober \cite{BL} studied the case when $n$ is odd and $l=nk$. In this special case, both $X_W$ and $Y_W$ are Calabi-Yau and of the same dimension, hence form a double mirror pair in the strict sense. They showed that the expected equality $E_{\mathrm{st}}(X_W)=E_{\mathrm{st}}(Y_W)$ of stringy $E$-functions actually holds in this case.

\smallskip

In general, when the dimension of $W$ is arbitrary or when $V$ is even-dimensional, $X_W$ and $Y_W$ are double mirrors only in a generalized sense: they are not necessarily Calabi-Yau, and they may have different dimensions. Therefore one cannot hope to have a genuine equality between their stringy $E$-functions. Nevertheless, they are related to each other in an interesting way. 

\smallskip

Before we state the main results of this paper, we make a few comments on the main differences between the even- and odd-dimensional cases. It has been observed in \cite{BL} that some issues occur in the even-dimensional case. The first issue is that when $n$ is even, it is impossible to choose $W$ such that $X_W$ and $Y_W$ are Calabi-Yau simultaneously. The second issue, which is more essential, is that the stringy $E$-function of the Pfaffian $\Pf(2k,V)$ is not a polynomial.

\smallskip

We explain how to address the second issue. To compute the stringy $E$-function, we need to choose a log resolution of $\Pf(2k,V)$. The log resolution $\pi:\widehat{\operatorname{Pf}(2k,V)}\rightarrow \operatorname{Pf}(2k,V)$ we consider in this paper is given by the \textit{space of complete skew forms of rank $\leq 2k$} (see Definition \ref{def-log-resolution}). The discrepancies of this resolution are computed in the following theorem.

\begin{theorem}[= Theorem \ref{usual discrepancy}]
	The discrepancies of the log resolution $\pi:\widehat{\operatorname{Pf}(2k,V)}\rightarrow \operatorname{Pf}(2k,V)$ are given by the equation
	\begin{align*}
		K_{\widehat{\operatorname{Pf}(2k,V)}}=\pi^* K_{\operatorname{Pf}(2k,V)} + \sum_{j}\alpha_{j,k,n}D_j
	\end{align*}
	where $\alpha_{j,k,n}=2j^2-j(n-2k)-1$, for $j=2,3,...,\frac{n-2}{2}$. See Theorem \ref{usual discrepancy} for the precise definition of $D_j$.
\end{theorem}

As noted above, these discrepancies yield a non-polynomial stringy $E$-function. \textbf{\textcolor{red}{One is then forced to consider a modified notion of discrepancy.}}

\begin{definition}[= Definition \ref{def-modified-discrepancies}]
	We define the \textit{modified discrepancy} of the log resolution $\pi:\widehat{\operatorname{Pf}(2k,V)}\rightarrow \operatorname{Pf}(2k,V)$ as
	\begin{align*}
		\widetilde{\alpha}_{j,k,n}=2j^2-j(n-2k+1)+\frac{n-2k-2}{2}
	\end{align*}
	which differs from the usual discrepancy $\alpha_{j,k,n}$ by a linear polynomial $j-\frac{n-2k}{2}$.
\end{definition}

We show in Theorem \ref{modified E-function} that the \textit{modified stringy $E$-function} of $\Pf(2k,V)$ computed from the modified discrepancies above is a genuine polynomial.

\smallskip

The idea of modifying the discrepancies originates from \cite{BL} and was further explored by Borisov and Wang in \cite{BW} for Clifford double mirrors. However, due to the difficulty of finding a Zariski locally trivial log resolution in the setting of Clifford double mirrors, only partial results at the level of stringy Euler characteristics (which is the specialization of the stringy $E$-function when $u,v\rightarrow 1$) were obtained in their work. The result presented in this paper is the first example establishing equalities at the level of stringy Hodge numbers, after the modification of the discrepancies.

\begin{theorem}[= Theorem \ref{thm-comparison}]
Let $n$ be an even integer, and let $X_W$ and $Y_W$ be the Pfaffian double mirrors corresponding to a generic $l$-dimensional subspace $W\subseteq \wedge^2 V^{\vee}$. We have the following relation between the modified stringy $E$-functions of $X_W$ and $Y_W$:
\begin{align*}
	q^{(n-1)k}\widetilde{E}_{\mathrm{st}}(Y_W)-q^{l}\widetilde{E}_{\mathrm{st}}(X_W)=\frac{q^{l}-q^{(n-1)k}}{q-1}\binom{n/2}{k}_{q^2}
\end{align*}
where $q=uv$, and $\binom{n/2}{k}_{q^2}=\prod_{j=k+1}^{\frac{n}{2}}\frac{q^{2j}-1}{q^{2j-2k}-1}$ is the $q$-binomial coefficient.
\end{theorem}

Applying a similar argument to the odd-dimensional cases (without modifying the discrepancies), we obtain the following result, generalizing a previous result of Borisov and Libgober.

\begin{theorem}[= Theorem \ref{thm-equality odd dim}]
	Let $n$ be an odd integer, and let $X_W$ and $Y_W$ be the Pfaffian double mirrors corresponding to a generic $l$-dimensional subspace $W\subseteq \wedge^2 V^{\vee}$. Then we have the following relation between the stringy $E$-functions of $X_W$ and $Y_W$:
	\begin{align*}
	q^{nk}E_{\mathrm{st}}(Y_W)-q^{l}E_{\mathrm{st}}(X_W)=\frac{q^{l}-q^{nk}}{q-1}\binom{(n-1)/2}{k}_{q^2}
\end{align*}
\end{theorem}

\subsection{Modified \texorpdfstring{$E$}{E}-functions as stringy \texorpdfstring{$E$}{E}-functions of pairs}

The modified $E$-functions introduced in this paper admit a speculative geometric interpretation as the stringy $E$-functions for ideal sheaf pairs. Let $X$ be a variety with log-terminal singularities, and let $\mathcal{I}$ be an ideal sheaf on $X$. The stringy $E$-function of the pair $(X,\mathcal{I})$ is defined in the usual way, except that the discrepancies are defined as
\begin{align*}
    \widetilde{a_j} = a_j - \mathrm{ord}_{D_j}(\mathcal{I})
\end{align*}
where $a_j$ is the usual discrepancy of the exceptional divisor $D_j$, and the extra term comes from the vanishing order of the ideal sheaf $\mathcal{I}$ along $D_j$. A more general definition can be found in \cite[Appendix A]{HQZ}. The independence of the choice of the log resolution is also proved in loc. cit.

In our situation, the ambient variety $X$ is the Pfaffian variety $\Pf(2k,V)$, and the ideal sheaf $\mathcal{I}$ is the ideal sheaf of the next lower-rank Pfaffian locus $\Pf(2k-2,V)$ inside $\Pf(2k,V)$. The computation of the relevant vanishing orders can be found in section \ref{subsec-pairs}. One immediately observes that they coincide with the correction terms in our formula for the modified discrepancies. Therefore we have the equality
\begin{align*}
	\widetilde{E}_{\mathrm{st}}(\Pf(2k,V))=E_{\mathrm{st}}(\Pf(2k,V),\mathcal{I}_{\Pf(2k-2,V)}).
\end{align*}
The author is not aware of a conceptual explanation for this equality beyond the direct computation. It would be interesting to understand the geometric meaning of the modified $E$-functions in a more conceptual way. However, at this moment, it remains a mysterious phenomenon.

\subsection{Stringy \texorpdfstring{$E$}{E}-functions, derived categories, and Homological Projective Duality}

It has been a long-standing problem to develop Hodge theory for generalized varieties (e.g., noncommutative algebraic varieties, or more generally, categorical resolutions of singular varieties). One of the motivations behind this is the observation that the derived category of a smooth variety shares many similarities to its classical Hodge-theoretic counterpart. Examples include the similarity between Orlov's formula for the semiorthogonal decomposition of the derived category of blow-ups and the formula for the stringy $E$-function of blow-ups. 

\smallskip

This problem is also important from the point of view of mirror symmetry, as the classical mirror symmetry predicts the equality between Hodge numbers of a double mirror pair, while the homological mirror symmetry predicts the equivalence of derived categories. It is then natural to study the relationship between these two aspects.

\smallskip

In the case of the Pfaffian varieties considered in this paper, the derived equivalence is realized by Kuznetsov's Homological Projective Duality. One of the inputs of Homological Projective Duality is a Lefschetz decomposition on the appropriate derived category, which plays a role similar to the classical Lefschetz decomposition in Hodge theory. The Homological Projective Duality for Pfaffian varieties in even-dimensional cases remains an open problem (for partial results in this direction, see Rennemo-Segal \cite{RS} and Pirozhkov \cite{Pirozhkov}). One of the main difficulties is constructing the categorical crepant resolutions of Pfaffians and their Lefschetz decompositions. Theorem \ref{thm-comparison}, the main result of this paper, allows us to make predictions on the shape of the Lefschetz decompositions of such categorical resolutions. More precisely, we make the following conjecture, which we hope to address in future work.

\begin{conjecture}[= Conjecture \ref{conj-lefschetz}]
The categorical crepant resolution $\widetilde{D}^b(\Pf(2k,V^{\vee}))$ of the Pfaffian variety $\Pf(2k,V^{\vee})$ has a non-rectangular Lefschetz decomposition of the form
	\begin{align*}
		\widetilde{D}^b(\Pf(2k,V^{\vee}))=\Big\langle \mathcal{A}_{0}, \mathcal{A}_{1}(1),\cdots,&\mathcal{A}_{nk-\frac{n}{2}-1}(nk-\frac{n}{2}-1),\\&\mathcal{A}_{nk-\frac{n}{2}}(nk-\frac{n}{2}),\cdots,\mathcal{A}_{nk-1}(nk-1) \Big\rangle
	\end{align*}
	where $\mathcal{A}_0=\cdots=\mathcal{A}_{nk-\frac{n}{2}-1}$ are $nk-\frac{n}{2}$ blocks of size $\binom{n/2}{k}$, and $\mathcal{A}_{nk-\frac{n}{2}}=\cdots=\mathcal{A}_{nk-1}$ are $n/2$ blocks of size $\binom{n/2-1}{k}$. Similarly, the dual Lefschetz decomposition of the dual $\Pf(n-2k,V)$ has the form
	\begin{multline*}
			\widetilde{D}^b(\Pf(n-2k,V))=\Big\langle \mathcal{B}_{\frac{n^2}{2}-nk-1}(-\frac{n^2}{2}+nk+1),\\
			\cdots \mathcal{B}_{\frac{n(n-1)}{2}-nk}(-\frac{n(n-1)}{2}+nk),\\
			\mathcal{B}_{\frac{n(n-1)}{2}-nk-1}(-\frac{n(n-1)}{2}+nk+1),\cdots,\mathcal{B}_{1}(-1),\mathcal{B}_0\Big\rangle
		\end{multline*}
	where $\mathcal{B}_0=\cdots=\mathcal{B}_{\frac{n(n-1)}{2}-nk-1}$ are $\frac{n(n-1)}{2}-nk$ blocks of size $\binom{n/2}{n/2-k}$, and $\mathcal{B}_{\frac{n(n-1)}{2}-nk}=\cdots=\mathcal{B}_{\frac{n^2}{2}-nk-1}$ are $n/2$ blocks of size $\binom{n/2-1}{n/2-k}$.
\end{conjecture}

After passing to the linear sections $X_W\subseteq \Pf(2k,V^{\vee})$ and $Y_W\subseteq \Pf(2k,V^{\vee})$, the appropriately defined categorical resolutions have semiorthogonal decompositions that are closely related to the Lefschetz decompositions of the ambient Pfaffian varieties, as predicted by Homological Projective Duality. In particular, when the dimension of $W$ varies, the ambient components (that are inherited from the ambient Pfaffian varieties) in the decomposition disappear on one side and reappear on the other side. We provide evidence for this conjecture by analyzing the behavior of the stringy $E$-functions of $X_W$ and $Y_W$ as the dimension of $W$ varies. For further details, see section \ref{sec-comments}.

\begin{remark}
	In \cite{RS}, Rennemo and Segal proved Homological Projective Duality for odd-dimensional Pfaffian varieties using the noncommutative resolutions of quotient singularities constructed by Špenko and van den Bergh \cite{SvdB}, together with matrix factorization techniques. They also obtained partial results in the even-dimensional setting. A major obstruction to establishing Pfaffian HPD in even dimensions is that the Špenko–van den Bergh noncommutative resolution is no longer crepant. Nevertheless, Rennemo and Segal proposed a conjectural HPD category in this setting; see \cite[eq.(5.2)]{RS}. As Ed Segal pointed out to the author, the sizes of the blocks appearing in their conjectural category agree precisely with those predicted by our results. The contribution of the present paper is to provide an independent, Hodge-theoretic numerical derivation of these block sizes from the modified stringy E-function, without relying on derived-categorical input.
\end{remark}

\begin{remark}
	Satriano and Usatine \cite{SW} constructed geometric crepant resolutions of log-terminal singularities by Artin stacks. From the perspective of mirror symmetry, the categorical crepant resolution should be equivalent to the category of B-branes on such an Artin stack crepant resolution, provided that the latter is appropriately defined. However, such a definition is not yet available in the literature and seems difficult to obtain, even for the simplest examples.
\end{remark}

\subsection{Organization of the paper}

The paper is organized as follows. In section \ref{sec-background} we include background on Pfaffian varieties, Pfaffian double mirrors and stringy $E$-functions, and fix the notation that will be used throughout this paper. In section \ref{sec-log resolution} we describe a Zariski locally trivial log resolution of $\Pf(2k,V)$ in terms of complete skew forms, and compute the discrepancies of the exceptional divisors. In section \ref{sec-modified stringy E-function of pfaffian} we explain the reason why it is necessary to modify the discrepancy in the even-dimensional case, and compute the modified stringy $E$-function of $\Pf(2k,V)$. In section \ref{sec-comparison} we compare the modified stringy $E$-function of the Pfaffian double mirrors, in both even and odd-dimensional cases (thereby generalizing the result of Borisov-Libgober \cite{BL}). The relationship between our results and Homological Projective Duality for Pfaffian varieties is included in section \ref{sec-comments}. Finally, we include auxiliary technical results that are used in the proof of the main result Theorem \ref{thm-comparison} in Appendices \ref{appendix-hypergeometric} and \ref{appendix-stringy E-function of cuts}.

\subsection{Acknowledgements}

The author would like to thank his advisor Lev Borisov for suggesting this problem, and for his constant support, helpful discussions and useful comments throughout the preparation of this paper. The author also thanks Ed Segal for a useful conversation regarding his paper \cite{RS}. The author is also grateful to the anonymous referees for helpful comments and suggestions that improved the presentation of this paper.

\section{Background on Pfaffian varieties and stringy \texorpdfstring{$E$}{E}-functions}\label{sec-background}

In this section, we provide background knowledge on Pfaffian varieties, Pfaffian double mirrors and stringy $E$-functions. We also fix the notation that will be used throughout this paper.

\subsection{Pfaffian varieties}

	Let $V$ be a complex vector space of  dimension $n\geq 4$. For $k=1,2,\cdots,\floor*{\frac{n}{2}}$, we define the Pfaffian variety\footnote{In some references, this is called a \textit{generalized} Pfaffian variety.} $\Pf(2k,V)$ to be the subvariety of $\PP(\wedge^2 V^{\vee})$ consisting of nontrivial skew forms on $V$ whose rank does not exceed $2k$. After fixing a basis for the vector space $V$, the Pfaffian variety $\Pf(2k,V)$ could also be seen as the space of skew-symmetric $n\times n$ matrices of rank at most $2k$, modulo scalar multiplication. We sometimes simply write $\Pf(2k,n)$ when there is no risk of confusion.
	
	\smallskip
	
	There is a natural stratification of the Pfaffian variety as a disjoint union of locally closed subvarieties
	\begin{align*}
		\Pf(2k,V)=\bigsqcup_{1\leq i\leq k}\Pf^{\circ}(2i,V)
	\end{align*}
	where the stratum $\Pf^{\circ}(2i,V)$ is the locus of skew forms of rank exactly equal to $2i$.
	
	The Pfaffian variety $\Pf(2k,V)$ is singular in general, except for the extremal cases $k=1$, where $\Pf(2,V)$ is naturally isomorphic to the Grassmannian $G(2,V)$, and the case $k=\floor*{\frac{n}{2}}$, where $\Pf(n,V)$ is the whole space $\PP(\wedge^2 V^{\vee})$. For $1<k<\floor*{\frac{n}{2}}$, the singular locus of $\Pf(2k,V)$ is the subvariety $\Pf(2k-2,V)$ of forms of the next lower ranks.

\begin{proposition}
	The Pfaffian variety $\Pf(2k,V)$ is Gorenstein, of dimension $2nk-2k^2-k-1$, and its canonical class is given by $-nk\xi$, where $\xi$ denotes the pullback of the hyperplane section $\mathcal{O}_{\PP(\wedge^2 V^{\vee})}(1)$ on $\PP(\wedge^2 V^{\vee})$. 
\end{proposition}
\begin{proof}
	The first statement is a simple dimension count. The canonical class of the Pfaffian variety could be computed from the standard Koszul resolution and Serre duality.
\end{proof}

\subsection{Pfaffian double mirrors}\label{subsec-generalized pfaffian double mirrors}

In this subsection we introduce the construction of the Pfaffian double mirrors.

Assume $k<\floor*{\frac{n}{2}}$. Let $W$ be a generic subspace of $\wedge^2 V^{\vee}$ of dimension $l$. Here, and throughout this paper, the generic condition means that all the non-empty intersections $\PP W^{\perp}\cap \Pf(2i, V^{\vee})$, $i=1,\cdots,k$ and $\PP W\cap \Pf(2i, V)$, $i=1,\cdots,\floor*{\frac{n}{2}}-k$ are transverse. In particular, every such intersection is either empty or smooth of the expected dimension. We define
\begin{itemize}
	\item $X_W$ to be the complete intersection of $\PP W^{\perp}$ and the Pfaffian variety $\Pf(2k,V^{\vee})$ in $\PP(\wedge^2 V)$.
	\item $Y_W$ to be the complete intersection of $\PP W$ and the Pfaffian variety $\Pf(2\floor*{\frac{n}{2}}-2k,V)$ in $\PP(\wedge^2 V^{\vee})$.
\end{itemize}
It is apparent from the definition that the roles of $X_W$ and $Y_W$ are switched if we interchange the data $(V,k,W)$ with $(V^{\vee},\floor*{\frac{n}{2}}-k,W^{\perp})$.

\smallskip

The main difference between the even- and odd-dimensional cases is the following. In the case where $n$ is odd, it is possible to make $X_W$ and $Y_W$ have the same dimension and Calabi-Yau simultaneously (indeed we achieve this by taking $\dim W=nk$). Therefore $X_W$ and $Y_W$ form a double mirror pair in the strict sense. However, it is impossible to do this when $n$ is even, and we only have a double mirror pair in a generalized sense. We have the following result.

\begin{proposition}
	Let $n$ be even. The variety $X_W$ is of dimension $2nk-l-2k^2-k-2$, and its canonical class is $(l-nk)\xi$. Similarly, the variety $Y_W$ is of dimension $l-2k^2+k-2$, and its canonical class is $(nk-\frac{n}{2}-l)\xi$.
\end{proposition}
\begin{proof}
	The variety $X_W$ can be equivalently defined as the set of forms $y\in\Pf(2k,V^{\vee})$ such that $\langle y, \alpha\rangle=0$ for any $\alpha\in W$, and can therefore be viewed as the complete intersection of $l$ hyperplanes with the Pfaffian variety $\Pf(2k,V)$ inside $\PP(\wedge^2 V)$. Thus a dimension count and the adjunction formula yield the statement for $X_W$. Interchanging $(V,k,W)$ with $(V^{\vee},\frac{n}{2}-k,W^{\perp})$ we get the statement for $Y_W$.
\end{proof}

\begin{example}
	(1) When $l=nk$, the varieties $X_W$ and $Y_W$ are of dimension $nk-2k^2-k-2$ and $nk-2k^2+k-2$ respectively. Furthermore, their canonical classes are $0$ and $-\frac{n}{2}\xi$ respectively. Therefore $X_W$ is Calabi-Yau while $Y_W$ is Fano.
	
	(2) When $l=(n-1)k$, the varieties $X_W$ and $Y_W$ are both of dimension $nk-2k^2-2$, and they have canonical classes $-k\xi$ and $(k-\frac{n}{2})\xi$ respectively. Hence in this case both $X_W$ and $Y_W$ are Fano.
\end{example}

\subsection{Stringy \texorpdfstring{$E$}{E}-functions}

Let $X$ be a complex algebraic variety. The Hodge-Deligne polynomial (or $E$-polynomial) of $X$ is defined as
	\begin{align*}
		E(X;u,v):=\sum_{p,q}\left(\sum_{i}(-1)^{i} h^{p,q}(H_c^i(X,\C))\right) u^p v^q
	\end{align*}
	where $H_c^i(X,\C)$ is the cohomology with compact support of $X$ together with its natural mixed Hodge structure. Two basic properties are additivity and multiplicativity.
	
	\begin{itemize}
		\item (Additivity) If $X=\sqcup_i X_i$ is a stratification of $X$ by Zariski locally closed subsets $X_i$, then the $E$-polynomial $E(X)$ is additive on the strata, i.e., $E(X)=\sum_i E(X_i)$.
		\item (Multiplicativity) If $\pi: Y\rightarrow X$ is a Zariski locally trivial fibration with fiber $F$, then $E(Y)=E(X)\cdot E(F)$.
	\end{itemize}

\begin{example}
	(1) The $E$-polynomial of the projective space $\PP^n$ is $E(\PP^n)=\frac{(uv)^{n+1}-1}{uv-1}$.
	
	(2) The $E$-polynomial of the Grassmannian $G(k,n)$ is
	\begin{align*}
		E(G(k,n))=\binom{n}{k}_{uv}:=\frac{\prod_{j=0}^{k-1}(1-(uv)^{n-j})}{\prod_{j=0}^{k-1}(1-(uv)^{j+1})}
	\end{align*}
	Here $\binom{n}{k}_{uv}$ is the $q$-binomial coefficient\footnote{See Appendix \ref{appendix-hypergeometric} for a precise definition. We will use this notation throughout this paper.}.
\end{example}

As we have discussed in section \ref{sec-intro}, the usual Hodge numbers do not behave well in the context of mirror symmetry. The correct replacement for the ordinary $E$-polynomials is the notion of stringy $E$-function introduced by Batyrev \cite{Batyrev}.

\begin{definition}
	Let $X$ be a singular variety with at worst log-terminal singularities. Let $\pi:\widehat{X}\rightarrow X$ be a log resolution and $\{D_i\}_{i\in I}$ be the set of exceptional divisors. Denote $D_J^{\circ}=(\bigcap_{j\in J}D_j)\backslash(\bigcup_{j\not\in J}D_j)$. The \textit{stringy $E$-function} of $X$ is given by
	\begin{align*}
		E_{\mathrm{st}}(X;u,v):=\sum_{J\subseteq I}E(D_{J}^{\circ})\prod_{j\in J}\frac{uv-1}{(uv)^{\alpha_j+1}-1},
	\end{align*}
	where $\alpha_j$ denotes the discrepancy of the exceptional divisor $D_j$, defined by the equation $K_{\widehat{X}}=\pi^* K_{X} + \sum_{j}\alpha_j D_j$.
\end{definition}

Although the definition involves a choice of a log resolution $\pi:\widehat{X}\rightarrow X$, it is proved by using the technique of motivic integration that its stringy $E$-function is indeed independent of such a choice.

\smallskip

It is worth mentioning that even for relatively simple varieties the stringy $E$-function is not necessarily a polynomial. Indeed, it is still an open problem under what conditions the stringy $E$-function is a polynomial. An obvious sufficient condition is that $X$ admits a crepant resolution, in which case the stringy $E$-function is the ordinary $E$-polynomial of the crepant resolution.

\smallskip

The stringy $E$-function is particularly easy to compute when $X$ admits a Zariski locally trivial log resolution.

\begin{proposition}
	Let $\pi:\widehat{X}\rightarrow X$ be a Zariski locally trivial log resolution. Define the local contribution of a point $x\in X$ to the stringy $E$-function $E_{\mathrm{st}}(X)$ by 
\begin{align*}
	S(x;u,v):=\sum_{J\subseteq I}E(D_{J}^{\circ}\cap\pi^{-1}(x))\prod_{j\in J}\frac{uv-1}{(uv)^{\alpha_j+1}-1}.
\end{align*}
Then the stringy $E$-function $E_{\mathrm{st}}(X)$ is
\begin{align*}
	E_{\mathrm{st}}(X)=\sum_{i}E(X_i;u,v)S(x\in X_i;u,v),
\end{align*}
where $X=\sqcup X_i$ is the stratification of $X$ into the strata where $S$ is a constant.
\end{proposition}

\section{Log resolutions of Pfaffian varieties}\label{sec-log resolution}

In this section, we describe a log resolution of the Pfaffian variety $\Pf(2k,V)$, and compute the discrepancies of its exceptional divisors.

\begin{definition}\label{def-log-resolution}
	Let $V$ be a complex vector space of dimension $n$, and $k$ be a positive integer with $k\leq \floor*{\frac{n}{2}}$. We define the space $\widehat{\Pf(2k,V)}$ of \textit{complete skew forms} of rank at most $2k$ as the iterated blow-up of the Pfaffian variety $\Pf(2k,V)$ along the subvarieties $\Pf(2,V)$, $\Pf(4,V)$, and so on, up to $\Pf(2k-2,V)$.
\end{definition}

\begin{proposition}\label{prop-complete skew forms}
	The variety $\widehat{\Pf(2k,V)}$ is smooth, and is a log resolution of the Pfaffian variety $\Pf(2k,V)$. Furthermore, its points are in 1-1 correspondence with flags
	\begin{align*}
		0\subseteq F^0\subseteq \cdots\subseteq F^l = V
	\end{align*}
	where $F^0$ is an $(n-2k)$-dimensional subspace, together with nondegenerate forms $w_i$ on the quotient $F^{i+1}/F^i$. The natural map $\pi:\widehat{\Pf(2k,V)}\rightarrow \Pf(2k,V)$ is then defined by mapping the data described above to the form $w_{l-1}$ considered as a skew form on $V$.
\end{proposition}
\begin{proof}
	See \cite{resolution} and \cite[Proposition 3.4]{BL}.
\end{proof}

\begin{remark}\label{rem-non-log-resolution}
	There is a simpler but non-log resolution of the singularities of $\Pf(2k,V)$. Consider the space
	\begin{align*}
		\left\{ (w,V_1): w\in\Pf(2k,V),\ V_1\subseteq\ker w,\ \dim V_1=n-2k \right\}.
	\end{align*}
	This space can be identified with the projective bundle $\pi: \PP(\wedge^2 Q)\rightarrow G(n-2k,V)$ where $Q$ is the universal quotient bundle of $G(n-2k,V)$. Then we have a diagram
	\begin{align*}
		\xymatrix{
		 \PP(\wedge^2 Q^{\vee})\ar[r]^{\pi}\ar[d]^{p} & G(n-2k,V) \\
		 \Pf(2k,V)
		}
	\end{align*}
	where the maps $p$ and $\pi$ are projections onto the two components. It is clear that $p$ is an isomorphism on the open locus $\Pf^{\circ}(2k,V)$ of forms of rank exactly $2k$, hence it provides a resolution of the singularities of $\Pf(2k,V)$.
\end{remark}

The most important property of $\widehat{\Pf(2k,V)}$, which is the key ingredient for computing the stringy $E$-function of $\Pf(2k,V)$, is as follows.

\begin{proposition}\label{prop-zariski locally trivial}
The fiber of $\pi:\widehat{\Pf(2k,V)}\rightarrow \Pf(2k,V)$ over a form $w\in\Pf^{\circ}(2i,V)$ of rank $2i$ is canonically isomorphic to $\widehat{\Pf(2k-2i,\ker w)}$. Furthermore, this resolution is Zariski locally trivial over each stratum $\Pf^{\circ}(2i,V)\subseteq\Pf(2k,V)$.
\end{proposition}
\begin{proof}
	By definition of $\pi:\widehat{\Pf(2k,V)}\rightarrow \Pf(2k,V)$, the fiber over $w\in\Pf^{\circ}(2i,V)$ consists of flags
	\begin{align*}
		0\subseteq F^0\subseteq \cdots\subseteq F^l = V
	\end{align*}
	where the dimensions of $F^0$ and $F^{l-1}$ are $n-2k$ and $n-2i$ respectively, together with choices of nondegenerate forms $w_i$ on each quotient $F^{i+1}/F^i$. In particular, $w$ is recovered by considering $w_{l-1}$ as a form on $F^{l}=V$, and $F^{l-1}$ is exactly the kernel of $w$. By forgetting the last vector space $F^{l}=V$ in the flag, we obtain a bijection between the fiber $\pi^{-1}(w)$ and the space of complete forms $\widehat{\Pf(2k-2i,\ker w)}$.

	\smallskip

	Now denote the restriction $\pi_i:=\pi|_{\pi^{-1}(\Pf^{\circ}(2i,V))}:\pi^{-1}(\Pf^{\circ}(2i,V))\rightarrow \Pf^{\circ}(2i,V)$. Our goal is to prove that $\pi_i$ is Zariski locally trivial. The universal skew form $\omega$ on $\Pf^{\circ}(2i,V)$ defines a morphism of vector bundles
	\begin{align*}
		\omega: V\otimes \mathcal{O}_{\Pf^{\circ}(2i,V)} \rightarrow V^{\vee}\otimes \mathcal{O}_{\Pf^{\circ}(2i,V)}(1).
	\end{align*}
	Denote its kernel by $\mathcal{K}_i$, which is a vector bundle of rank $n-2i$. Then we have a natural sequence
	\begin{align*}
		0\rightarrow \mathcal{K}_i\rightarrow V\otimes \mathcal{O}_{\Pf^{\circ}(2i,V)} \rightarrow \mathcal{Q}_i \rightarrow 0
	\end{align*}
	where $\mathcal{Q}_i$ is the quotient bundle of rank $2i$. The universal skew form $\omega$ induces a nondegenerate skew form $\omega: \wedge^2 \mathcal{Q}_i \rightarrow \mathcal{O}_{\Pf^{\circ}(2i,V)}(1)$ on the quotient bundle $\mathcal{Q}_i$.

	\smallskip

	Now for any point $x_0\in \Pf^{\circ}(2i,V)$, we can take a sufficiently small Zariski open neighborhood $U$ that trivializes $\mathcal{Q}_i$, $\mathcal{K}_i$, and over which the short exact sequence above splits. We then choose a frame of $\mathcal{Q}_i$ over $U$ such that the matrix of $\omega$ is the standard symplectic form. Then we lift them to sections of $V\otimes \mathcal{O}_U$. On the other hand we can also choose a basis of $\mathcal{K}_i$ over $U$. Together they form a basis of $V\otimes \mathcal{O}_U$, which gives a morphism $g: U\rightarrow \mathrm{GL}(V)$.

	\smallskip

	Finally, since the iterative blow-up construction is $\mathrm{GL}(V)$-equivariant, the morphism $\Phi: U\times \pi^{-1}(x_0)\rightarrow \pi^{-1}(U)$ defined by $\Phi(x,z) = g(x)z$ gives the desired local trivialization.
\end{proof}

Now, we compute the discrepancies of the log resolution when $V$ is even-dimensional.

\begin{theorem}\label{usual discrepancy}
	Let $V$ be a vector space of even dimension $n$ and let $k\leq\frac{n-2}{2}$ be a positive integer. Denote by $D_j$ the exceptional divisor of the log resolution $\widehat{\operatorname{Pf}(2k,V)}\rightarrow \operatorname{Pf}(2k,V)$ that corresponds to the locus of complete forms with a subspace of dimension $2j$ present in the flag. The discrepancy of $D_j$ is given by
	\begin{align*}
		\alpha_{j,k,n}=2j^2-j(n-2k)-1
	\end{align*}
	for $j=2,3,...,\frac{n-2}{2}$.
\end{theorem}

\begin{proof}
To simplify notation, we denote $r=n-2k$, and denote the discrepancies $\alpha_{j,k,n}$ simply by $\alpha_{j}$ since we will be working with fixed $k$ and $n$ throughout the proof. By Proposition \ref{prop-complete skew forms}, the points of $\widehat{\operatorname{Pf}(2k,V)}$ are in 1-1 correspondence with flags
	\begin{align*}
		0\subseteq F^0\subseteq \cdots\subseteq F^l = V
	\end{align*}
where $\dim F^0=r$, together with nondegenerate forms on the quotients $F^{i+1}/F^i$. We consider a new space $Z$ defined by the following fiber product
	\begin{align*}
		\xymatrix{
		Z:=\widehat{\operatorname{Pf}(2k,V)}\times_{G(r,V)}\operatorname{Fl}(r-1,r,V)\ar[r]\ar[d]^{\mu} & \operatorname{Fl}(r-1,r,V)\ar[d] \\
		\widehat{\operatorname{Pf}(2k,V)}\ar[r] & G(r,V)
		}
	\end{align*}
	where $\operatorname{Fl}(r-1,r,V)$ is the two-step partial flag variety of subspaces of dimensions $r-1$ and $r$ on $V$, the bottom horizontal map is defined by mapping the flag above to the lowest dimensional space $F^0$, and the right vertical map is the canonical map. It is easy to see that $Z$ can be described as the space of extended flags
	\begin{align*}
		0\subseteq F^{-1}\subseteq F^0\subseteq \cdots\subseteq F^l = V
	\end{align*}
	where $\dim F^{-1}=r-1$ and $\dim F^0=r$, together with choices of nondegenerate forms on the quotients $F^{i+1}/F^i$ for $i\geq 0$.
	
	\smallskip
	
	The main idea of the proof is to compute the canonical class $K_Z$ of $Z$ in three different ways.
	
	\smallskip

\textbf{Step 1.}
	
	First we consider the diagram
	\begin{align*}
		\xymatrix{
		Z\ar[r]^{\pi_1}\ar[dr]_{\pi_3} & \mathbb{P}(\wedge^2 Q_{r,\operatorname{Fl}}^{\vee})\ar[d]^{\varphi} \\
		 & \operatorname{Fl}(r-1,r,V)
		}
	\end{align*}
	We now explain the meaning of the symbols and maps in this diagram. Here $Q_{r,\mathrm{Fl}}$ is the universal quotient bundle of the flag variety $\operatorname{Fl}(r-1,r,V)$ corresponding to the $r$-dimensional subspaces in the flags, defined by the sequence
	\begin{align*}
		0\rightarrow \mathcal{H}_r \rightarrow \operatorname{Fl}(r-1,r,V)\times V \rightarrow Q_{r,\mathrm{Fl}}\rightarrow 0
	\end{align*}
	where the intermediate term is the trivial bundle of rank $n$, and $\mathcal{H}_r$ is the rank $r$ subbundle whose fiber over a point of $\operatorname{Fl}(r-1,r,V)$ is given by the $r$-dimensional subspace in the corresponding flag. The projective bundle $\mathbb{P}(\wedge^2 Q_{r,\operatorname{Fl}}^{\vee})$ can be viewed as the space of skew forms on the vector bundle $Q_{r,\mathrm{Fl}}$. The map $\pi_1$ is then defined by mapping an extended flag
	\begin{align*}
		0\subseteq F^{-1}\subseteq F^0\subseteq \cdots\subseteq F^l = V
	\end{align*}
	with nondegenerate forms $w_i$ on $F^{i+1}/F^i$ to the form $w_l$ considered as a skew form on the quotient $F^{l}/F^{0}=V/F^{0}$. The definitions of $\pi_2$ and $\varphi$ are clear.
	
	\smallskip
	
	Now we compute the canonical class $K_Z$ of $Z$ from this diagram. Recall that the canonical class of a projective bundle $\pi:\PP\mathcal{E}\rightarrow X$ with $\rk\mathcal{E}=r+1$ is given by
	\begin{align*}
		K_{\PP\mathcal{E}}=-(r+1)\mathcal{O}_{\PP\mathcal{E}}(1)+\pi^*(K_X+\det(\mathcal{E})).
	\end{align*}
	Applying this formula to $\PP(\wedge^2 Q_{r,\mathrm{Fl}})^{\vee}\rightarrow\operatorname{Fl}(r-1,r,V)$, we obtain
	\begin{equation}\label{eqs-disc-1}
		\begin{aligned}
		K_Z&=\pi_1^* K_{\mathbb{P}\wedge^2 Q_{r,\operatorname{Fl}}^{\vee}}+\sum_{j=\frac{n-2k+2}{2}}^{\frac{n-2}{2}}\left(\frac{1}{2}(2j-r)(2j-r-1)-1\right)E_j\\
		&=\pi_1^*\Bigl(-\operatorname{rank}\wedge^2 Q_{r,\operatorname{Fl}}^{\vee}\cdot\mathcal{O}_{\mathbb{P}\wedge^2 Q_{r,\operatorname{Fl}}^{\vee}}(1)\\
			&\qquad +\varphi^*(K_{\operatorname{Fl}(r-1,r,V)}-c_1(\wedge^2 Q_{r,\operatorname{Fl}}^{\vee}))\Bigr)\\
		&\quad +\sum_{j=\frac{n-2k+2}{2}}^{\frac{n-2}{2}}\left(\frac{1}{2}(2j-r)(2j-r-1)-1\right)E_j\\
		&=\pi_3^*K_{\operatorname{Fl}(r-1,r,V)}+(n-r-1)\pi_3^* c_1(Q_{r,\operatorname{Fl}})\\
			&\quad -\frac{(n-r)(n-r-1)}{2}\pi_1^*\mathcal{O}_{\mathbb{P}\wedge^2 Q_{r,\operatorname{Fl}}^{\vee}}(1)\\
		&\quad +\sum_{j=\frac{n-2k+2}{2}}^{\frac{n-2}{2}}\left(\frac{1}{2}(2j-r)(2j-r-1)-1\right)E_j.
	\end{aligned}
	\end{equation}
	
	\smallskip
	
	\textbf{Step 2.}
	
	Next, we consider the following diagram
	\begin{align*}
		\xymatrix{
		Z\ar[r]^{\pi_2}\ar[dr]_{\pi_4} & \mathbb{P}(\wedge^2 Q_{r-1}^{\vee})\ar[d]^{\psi} \\
		 & G(r-1,V)
		}
	\end{align*}
	Here $Q_{r-1}$ is the universal quotient bundle on the Grassmannian $G(r-1,V)$, and $\mathbb{P}(\wedge^2 Q_{r-1}^{\vee})$ is the space of skew forms on $Q_{r-1}$. The definition of $\pi_2$ is similar to that of $\pi_1$, mapping an extended flag with nondegenerate forms $\{w_i\}_{i\geq 0}$ to the form $w_l$ considered as a skew form on $V/F^{-1}$. 
	
	\smallskip
	
	Similarly to Step 1, from this diagram we obtain
	\begin{equation}\label{eqs-disc-2}
		\begin{aligned}
			K_Z&=\pi_2^* K_{\mathbb{P}\wedge^2 Q_{r-1}^{\vee}}+\sum_{j=\frac{n-2k+2}{2}}^{\frac{n-2}{2}}\left(\frac{1}{2}(2j-r+1)(2j-r)-1\right)E_j\\
		&=\pi_2^*\Bigl(-\operatorname{rank}\wedge^2 Q_{r-1}^{\vee}\cdot\mathcal{O}_{\mathbb{P}\wedge^2 Q_{r,\operatorname{Fl}}^{\vee}}(1)\\
			&\qquad +\psi^*(K_{G(r-1,V)}-c_1(\wedge^2 Q_{r-1}^{\vee}))\Bigr)\\
		&\quad +\sum_{j=\frac{n-2k+2}{2}}^{\frac{n-2}{2}}\left(\frac{1}{2}(2j-r+1)(2j-r)-1\right)E_j\\
		&=\pi_4^*K_{G(r-1,V)}+(n-r)\pi_4^* c_1(Q_{r-1})\\
			&\quad -\frac{(n-r+1)(n-r)}{2}\pi_1^*\mathcal{O}_{\mathbb{P}\wedge^2 Q_{r-1}^{\vee}}(1)\\
		&\quad +\sum_{j=\frac{n-2k+2}{2}}^{\frac{n-2}{2}}\left(\frac{1}{2}(2j-r+1)(2j-r)-1\right)E_j.
		\end{aligned}
	\end{equation}
	
	\smallskip
	
	\textbf{Step 3.}

	Finally, we consider the third diagram
	\begin{align*}
		\xymatrix{
		Z\ar[r]^{\pi_3}\ar[d]_{\mu}\ar[dr]^{\pi_5} & \operatorname{Fl}(r-1,r,V)\ar[d] \\
		\widehat{\operatorname{Pf}(2k,V)}\ar[r] & G(r,V)
		}
	\end{align*}
	we have
	\begin{equation}\label{eqs-disc-3}
		\begin{aligned}
			K_Z&=\mu^* K_{\widehat{\operatorname{Pf}(2k,V)}}+\pi_3^*K_{\operatorname{Fl}(r-1,r,V)} - \pi_5^* K_{G(r,V)}\\
		&=\mu^*\left(\pi^* K_{\operatorname{Pf}(2k,V)}+\sum_{j}\alpha_j D_j\right)+\pi_3^*K_{\operatorname{Fl}(r-1,r,V)} - \pi_5^* K_{G(r,V)}\\
		&=-\frac{n(n-r)}{2}\mu^*\pi^*\mathcal{O}_{\mathbb{P}\wedge^2 V^{\vee}}(1)+\sum_{j}\alpha_j E_j+\pi_3^*K_{\operatorname{Fl}(r-1,r,V)} - \pi_5^* K_{G(r,V)}.
		\end{aligned}
	\end{equation}
	
	Taking the linear combination of the three equations \eqref{eqs-disc-1}, \eqref{eqs-disc-2}, and \eqref{eqs-disc-3} with coefficients $r-1$, $-r-1$ and $2$ respectively yields
	\begin{align*}
		0=&(r+1)\pi_3^* K_{\operatorname{Fl}(r-1,r,V)}+(r-1)(n-r-1)\pi_3^* c_1(Q_{r,\operatorname{Fl}})\\
			&-(r+1)\pi_4^* K_{G(r-1,V)}-(r+1)(n-r)\pi_4^* c_1(Q_{r-1})-2\pi_5^*K_{G(r,V)}\\
			&+\sum_{j=\frac{n-2k+2}{2}}^{\frac{n-2}{2}}\left(2\alpha_j+2-4j^2+2jr\right)E_j.
	\end{align*}
	We claim that the following relation holds in the Picard group of $Z$:
	\begin{align*}
		0=&(r+1)\pi_3^* K_{\operatorname{Fl}(r-1,r,V)}+(r-1)(n-r-1)\pi_3^* c_1(Q_{r,\operatorname{Fl}})\\
			&-(r+1)\pi_4^* K_{G(r-1,V)}-(r+1)(n-r)\pi_4^* c_1(Q_{r-1})-2\pi_5^*K_{G(r,V)}.
	\end{align*}
	To see this, we denote the natural maps $\operatorname{Fl}(r-1,r,V)\rightarrow G(r-1,V)$ and $\operatorname{Fl}(r-1,r,V)\rightarrow G(r,V)$ by $f$ and $g$ respectively. It then suffices to prove
	\begin{equation}\label{eqs-disc-4}
			\begin{aligned}
		0=&(r+1) K_{\operatorname{Fl}(r-1,r,V)}+(r-1)(n-r-1)c_1(Q_{r,\operatorname{Fl}})-(r+1)f^* K_{G(r-1,V)}\\
		&-(r+1)(n-r)f^* c_1(Q_{r-1})-2g^*K_{G(r,V)}.
	\end{aligned}
	\end{equation}
	Note that we have the following exact sequence
	\begin{align*}
		0\rightarrow\operatorname{Hom}(T_{r,\mathrm{Fl}},Q_{r,\mathrm{Fl}})\rightarrow T_{\mathrm{Fl}(r-1,r,V)}\rightarrow\operatorname{Hom}(T_{r-1,\mathrm{Fl}},T_{r,\mathrm{Fl}}/T_{r-1,\mathrm{Fl}})\rightarrow 0.
	\end{align*}
	Therefore, the canonical class is computed as
	\begin{align*}
		K_{\operatorname{Fl}(r-1,r,V)}&=n\cdot c_1(Q_{r,\mathrm{Fl}})+c_1(\operatorname{Hom}(T_{r-1,\mathrm{Fl}},T_{r,\mathrm{Fl}}))\\
			&\quad -c_1(\operatorname{Hom}(T_{r-1,\mathrm{Fl}},T_{r-1,\mathrm{Fl}}))\\
		&=n\cdot c_1(Q_{r,\mathrm{Fl}})-r\cdot c_1(T_{r-1,\mathrm{Fl}})+(r-1)\cdot c_1(T_{r,\mathrm{Fl}})\\
		&=(n-r+1)c_1(Q_{r,\mathrm{Fl}})+r\cdot c_1(Q_{r-1,\mathrm{Fl}}).
	\end{align*}
	Substituting this into the right side of \eqref{eqs-disc-4} and applying the well-known formula $K_{G(r,V)}=-n\cdot c_1(Q_r)$, we obtain the desired equality. 
	
	\smallskip
	
	Therefore we conclude that $\alpha_j=2j^2-jr-1=2j^2-j(n-2k)-1$.

\end{proof}

\section{Modified stringy \texorpdfstring{$E$}{E}-functions of Pfaffian varieties}\label{sec-modified stringy E-function of pfaffian}

Throughout this section, we assume that $n$ is even. In the previous section, we computed the discrepancies $\alpha_{j,k,n}$ of the exceptional divisors $D_j$ of the log resolution $\pi:\widehat{\operatorname{Pf}(2k,V)}\rightarrow \operatorname{Pf}(2k,V)$. This allows us to compute the stringy $E$-function of $\operatorname{Pf}(2k,V)$. In fact, we have the following formula for $E_{\mathrm{st}}(\operatorname{Pf}(2k,V))$. However, we note that this is not the \emph{correct} stringy $E$-function to establish the desired equality. 

\begin{theorem}\label{thm-ordinary stringy E-function}
    Let $V$ be an $n$-dimensional vector space. The stringy $E$-function of $\Pf(2k,V)$ is
    \begin{align*}
		E_{\mathrm{st}}(\operatorname{Pf}(2k,V))=\frac{(uv)^{nk}-1}{uv-1}\prod_{j=1}^{k}\frac{(uv)^{n+1-2j}-1}{(uv)^{2j}-1}
	\end{align*}
	for $1\leq k\leq \frac{n}{2}$.
\end{theorem}
Since this result will not be used in the remainder of the paper, we provide only a sketch of the proof; see the paragraph immediately before section \ref{subsec-pairs}. The proof is similar to that of Theorem \ref{modified E-function}.

\medskip

It is clear from the formula that the stringy $E$-function of the Pfaffian variety $\operatorname{Pf}(2k,V)$ is \textit{not} a polynomial, except for the extremal cases $k=1$ and $n/2$. For example, when $k=2$ and $n=6$, this formula gives
	\begin{align*}
		E_{\mathrm{st}}(\operatorname{Pf}(4,6))=\frac{(uv)^{12}-1}{uv-1}\cdot\frac{((uv)^{3}-1)((uv)^{5}-1)}{((uv)^2-1)((uv)^{4}-1)}.
	\end{align*}
To obtain a genuine polynomial, it is necessary to modify the discrepancy in the formula for the stringy $E$-function. 

\begin{definition}\label{def-modified-discrepancies}
	We define the \textit{modified discrepancy} of the log resolution $\pi:\widehat{\operatorname{Pf}(2k,V)}\rightarrow \operatorname{Pf}(2k,V)$ to be
	\begin{align*}
		\widetilde{\alpha}_{j,k,n}=2j^2-j(n-2k+1)+\frac{n-2k-2}{2}
	\end{align*}
	that differs from the usual discrepancy $\alpha_{j,k,n}$ calculated in Theorem \ref{usual discrepancy} by a linear polynomial $j-\frac{n-2k}{2}$.
\end{definition}

\begin{remark} 
	The modified discrepancy $\widetilde{\alpha}_{j,k,n}$ is strictly smaller than the usual discrepancy $\alpha_{j,k,n}$. It is well-known that in the process of a blow-up the discrepancy of the exceptional divisor is equal to the codimension of the center minus 1. Our modification of the discrepancy indicates that the subvarieties $\Pf(2i,V)\subseteq\Pf(2k,V)$ somehow behave as if they were of smaller codimension than what they actually are.
\end{remark}

We will use these modified discrepancies to compute the modified stringy $E$-function of the Pfaffian variety $\Pf(2k,V)$, which we denote by $\widetilde{E}_{\mathrm{st}}(\Pf(2k,V))$. The main result of this section is the following:

\begin{theorem}\label{modified E-function}
	Set $q=uv$. With the modified discrepancies described in Definition \ref{def-modified-discrepancies}, the modified stringy $E$-function of $\Pf(2k,V)$ is
	\begin{align*}
		\widetilde{E}_{\mathrm{st}}(\operatorname{Pf}(2k,V))=\frac{q^{(n-1)k}-1}{q-1}\binom{n/2}{k}_{q^2}.
	\end{align*}
\end{theorem}

We will prove this formula by induction. The following technical lemma is crucial to the induction argument.

\begin{lemma}\label{lem-technical indentities}
	For an even integer $n$ and any $1\leq k \leq \frac{n-2}{2}$, there holds
	\begin{align*}
		\sum_{1\leq i \leq \frac{n}{2}}e_{2i}g_{n-2i,n}\prod_{j=k-i+1}^{\frac{n-2i}{2}}\frac{q^{2j}-1}{q^{2j-2k+2i}-1}=\frac{q^{(n-1)k}-1}{q-1}\prod_{j=k+1}^{\frac{n}{2}}\frac{q^{2j}-1}{q^{2j-2k}-1}
	\end{align*}
	where $e_{2i}$ denotes the $E$-function of the variety of \textit{nondegenerate} skew forms on $\CC^{2i}$ up to scaling, and $g_{n-2i,n}:=\binom{n}{n-2i}_{q}$ denotes the $E$-function of the Grassmannian $G(n-2i,n)$.
\end{lemma}
\begin{proof}
	We prove this equality by induction on $n+k$. We omit the proof of the base case, since it is straightforward.
	
	\smallskip 
	
	Assume $1<k<\frac{n-2}{2}$ (the extremal cases $k=1$ and $k=\frac{n-2}{2}$ are again straightforward since there is only one nonzero term in the sum on the left side). The induction hypothesis gives
	\begin{align*}
		\sum_{1\leq i \leq \frac{n}{2}}e_{2i}g_{n-2i,n}\prod_{j=k-i}^{\frac{n-2i}{2}}\frac{q^{2j}-1}{q^{2j-2k+2+2i}-1}=\frac{q^{(n-1)(k-1)}-1}{q-1}\prod_{j=k}^{\frac{n}{2}}\frac{q^{2j}-1}{q^{2j-2k+2}-1}
	\end{align*}
	and
	\begin{align*}
		\sum_{1\leq i \leq \frac{n-2}{2}}e_{2i}g_{n-2-2i,n-2}\prod_{j=k-i+1}^{\frac{n-2-2i}{2}}\frac{q^{2j}-1}{q^{2j-2k+2i}-1}=\frac{q^{(n-3)k}-1}{q-1}\prod_{j=k+1}^{\frac{n-2}{2}}\frac{q^{2j}-1}{q^{2j-2k}-1}
	\end{align*}
	which correspond to the cases $(k-1,n)$ and $(k,n-2)$ respectively. We rewrite them as
	\begin{align*}
		\sum_{1\leq i \leq \frac{n}{2}}&\frac{q^{2k-2i}-1}{q^{n-2k+2}-1}e_{2i}g_{n-2i,n}\prod_{j=k-i+1}^{\frac{n-2i}{2}}\frac{q^{2j}-1}{q^{2j-2k+2i}-1} \\
		&=\frac{q^{(n-1)(k-1)}-1}{q-1}\prod_{j=k}^{\frac{n}{2}}\frac{q^{2j}-1}{q^{2j-2k+2}-1}
	\end{align*}
	and
	\begin{align*}
		\sum_{1\leq i \leq \frac{n-2}{2}}&\frac{(q^{n-2k}-1)(q^{n-2i-1}-1)}{(q^n-1)(q^{n-1}-1)}e_{2i}g_{n-2-2i,n-2}\prod_{j=k-i+1}^{\frac{n-2i}{2}}\frac{q^{2j}-1}{q^{2j-2k+2i}-1} \\
		&=\frac{q^{(n-3)k}-1}{q-1}\prod_{j=k+1}^{\frac{n-2}{2}}\frac{q^{2j}-1}{q^{2j-2k}-1}.
	\end{align*}
	Taking the linear combination of these two equations with coefficients
	\begin{align*}
		q^n(q^{n-2k+2}-1)\quad\text{and}\quad -\frac{q^{2k+1}(q^n-1)(q^{n-1}-1)}{q^{n-2k}-1}
	\end{align*}
	respectively yields the desired equality.

\end{proof}

\begin{proof}[Proof of Theorem \ref{modified E-function}]
	Consider the strata $\operatorname{Pf}^{\circ}(2i, V)\subseteq \operatorname{Pf}(2k,V)$. By Proposition \ref{prop-zariski locally trivial}, the resolution $\pi:\widehat{\operatorname{Pf}(2k,V)}\rightarrow \operatorname{Pf}(2k,V)$ is Zariski locally trivial over $\operatorname{Pf}^{\circ}(2i, V)$, hence the local contribution is
	\begin{align*}
		E(\operatorname{Pf}^{\circ}(2i, V))\widetilde{E}_{\mathrm{st}}(\operatorname{Pf}(2k-2i,\operatorname{ker}w))\frac{q-1}{q^{\widetilde{\alpha}_{\frac{n-2i}{2},k,n}+1}-1}.
	\end{align*}
	Note that $2j=n-2i$. By the induction hypothesis this equals
	\begin{align*}
		E(\operatorname{Pf}^{\circ}(2i, V))\frac{q^{(n-2i-1)(k-i)}-1}{q-1}\prod_{j=k-i+1}^{\frac{n-2i}{2}}\frac{q^{2j}-1}{q^{2j-2k+2i}-1}\cdot\frac{q-1}{q^{\widetilde{\alpha}_{\frac{n-2i}{2},k,n}+1}-1}.
	\end{align*}
	Since $\widetilde{\alpha}_{\frac{n-2i}{2},k,n}=(n-2i-1)(k-i)-1$, this equals
	\begin{align*}
		e_{2i}g_{n-2i,n}\prod_{j=k-i+1}^{\frac{n-2i}{2}}\frac{q^{2j}-1}{q^{2j-2k+2i}-1}.
	\end{align*}
	Therefore, the modified stringy $E$-function is
	\begin{align*}
		\widetilde{E}_{\mathrm{st}}(\operatorname{Pf}(2k,V))&=\sum_{1\leq i\leq k}e_{2i}g_{n-2i,n}\prod_{j=k-i+1}^{\frac{n-2i}{2}}\frac{q^{2j}-1}{q^{2j-2k+2i}-1}\\
		&=\sum_{1\leq i\leq \frac{n}{2}}e_{2i}g_{n-2i,n}\prod_{j=k-i+1}^{\frac{n-2i}{2}}\frac{q^{2j}-1}{q^{2j-2k+2i}-1}\\
		&=\frac{q^{(n-1)k}-1}{q-1}\prod_{j=k+1}^{\frac{n}{2}}\frac{q^{2j}-1}{q^{2j-2k}-1}.
	\end{align*}
\end{proof}

\begin{proof}[Sketch of the proof of Theorem \ref{thm-ordinary stringy E-function}]
	The proof is almost identical to the previous proof. The only difference is that we replace the modified discrepancies $\widetilde{\alpha}_{j,k,n}$ by the usual discrepancies $\alpha_{j,k,n}$ in the computation of the local contribution. The identity in Lemma \ref{lem-technical indentities} is replaced by
	\begin{align*}
		\sum_{1\leq i \leq \frac{n}{2}}e_{2i}g_{n-2i,n}\prod_{j=1}^{k-i}\frac{q^{n-2i+1-2j}-1}{q^{2j}-1}=\frac{q^{nk}-1}{q-1}\prod_{j=1}^{k}\frac{q^{n+1-2j}-1}{q^{2j}-1}
	\end{align*}
	and can also be proved by induction on $n+k$. More precisely, one replaces the factors of the inductive assumptions corresponding to $(k-1,n)$ and $(k,n-2)$ by
	\begin{align*}
		\frac{q^{2k-2i}-1}{q^{n-2k+1}-1}\ \text{ and }\  \frac{(q^{n-2k-1}-1)(q^{n-2i}-1)}{(q^n-1)(q^{n-1}-1)}
	\end{align*}
	and takes the linear combination with coefficients
	\begin{align*}
		q^n(q^{n-2k+1}-1)\ \text{ and }\  -\frac{q^{2k}(q^n-1)(q^{n-1}-1)}{q^{n-2k-1}-1}.
	\end{align*}
\end{proof}

\subsection{Modified stringy \texorpdfstring{$E$}{E}-functions as stringy \texorpdfstring{$E$}{E}-functions of pairs}\label{subsec-pairs}

We end this section with a discussion of a speculative geometric explanation of the modified $E$-functions. The definition of stringy $E$-functions can be generalized to pairs $(X,\mathcal{I})$ where $\mathcal{I}$ is an ideal sheaf on $X$. In this case, the stringy $E$-function is defined in the usual way, except that the discrepancies are defined as
\begin{align*}
    \widetilde{a_j} = a_j - \mathrm{ord}_{D_j}(\mathcal{I})
\end{align*}
where $a_j$ is the usual discrepancy of the exceptional divisor $D_j$, and the extra term comes from the vanishing order of the ideal sheaf $\mathcal{I}$ along $D_j$. A proof of the fact that this definition is independent of the choice of the log resolution can be found in \cite[Appendix A]{HQZ}.

\smallskip

The modified $E$-function of $\Pf(2k, n)$ we defined here is precisely the stringy $E$-function of the pair $(\Pf(2k, n), \mathcal{I}_{\Pf(2k-2, n)})$, where $\mathcal{I}_{\Pf(2k-2, n)}$ is the ideal sheaf of the next lower-rank Pfaffian locus inside $\Pf(2k, n)$. To see this, note that at a general point of the divisor $D_j$, the complete skew form is given by a flag $0\subseteq F^0 \subseteq K \subseteq V$ where the dimensions of $F^0$ and $K$ are $n-2k$ and $2j$ respectively, together with nondegenerate skew forms $\alpha$ and $\beta$ on the quotients $V/K$ and $K/F^0$. Choose a curve on the log resolution $\widehat{\Pf(2k, n)}$ that intersects $D_j$ transversely. After choosing a splitting, its image under $\pi$ can be written as
\begin{align*}
	\omega_t = \alpha + t\beta +O(t^2).
\end{align*}
Note that the ideal sheaf $\mathcal{I}_{\Pf(2k-2, n)}$ is generated by the Pfaffians of size $2k$, or equivalently by the coordinates of $\bigwedge^k \omega$. It is easy to see that the leading term of $\wedge^k \omega$ is given by $t^{k-\frac{n}{2}+j}\bigwedge^{\frac{n}{2}-j}\alpha \wedge\bigwedge^{k-\frac{n}{2}+j}\beta$. Therefore every generator of the pullback of the ideal sheaf vanishes to order $k-\frac{n}{2}+j$ along $D_j$. This matches with the difference between the modified and the usual discrepancies; see Definition \ref{def-modified-discrepancies}. Therefore we have the equality
\begin{align*}
	\widetilde{E}_{\mathrm{st}}(\Pf(2k,V))=E_{\mathrm{st}}(\Pf(2k,V),\mathcal{I}_{\Pf(2k-2,V)}).
\end{align*}

\medskip

\section{Comparison of stringy \texorpdfstring{$E$}{E}-functions of Pfaffian double mirrors}\label{sec-comparison}

In this section, we compare the modified stringy $E$-functions of the Pfaffian double mirrors $X_W$ and $Y_W$ for a generic subspace $W\subseteq\wedge^2 V^{\vee}$ of dimension $l$.

\begin{theorem}\label{thm-comparison}
Let $V$ be a complex vector space of even dimension $n$. Let $k$ be a positive integer with $k\leq\frac{n}{2}$ and $W\subseteq\wedge^2 V^{\vee}$ be a generic subspace of dimension $l$. Let $X_W$ and $Y_W$ be the Pfaffian double mirrors as defined in \S \ref{subsec-generalized pfaffian double mirrors}. Then we have the following relation between the modified stringy $E$-functions of $X_W$ and $Y_W$:
\begin{align*}
	q^{(n-1)k}\widetilde{E}_{\mathrm{st}}(Y_W)-q^{l}\widetilde{E}_{\mathrm{st}}(X_W)=\frac{q^{l}-q^{(n-1)k}}{q-1}\binom{n/2}{k}_{q^2}
\end{align*}
where $\binom{n/2}{k}_{q^2}=\prod_{j=k+1}^{\frac{n}{2}}\frac{q^{2j}-1}{q^{2j-2k}-1}$ is the $q$-binomial coefficient.
\end{theorem}

\medskip

The idea of the proof is to consider the universal hypersurface $H\subseteq \Pf(2k,V^{\vee})\times\PP W$ defined by
\begin{align*}
	H:=\left\{(w,\alpha)\in \Pf(2k,V^{\vee})\times\PP W: \langle w,\alpha\rangle =0 \right\}
\end{align*}
and the two projections $\pi_1:H\rightarrow \Pf(2k,V^{\vee})$ and $\pi_2:H\rightarrow \PP W$. The modified stringy $E$-function of $H$ can be computed in two different ways, which gives the desired equality. Here, the modified stringy $E$-function of $H$ is defined similarly to that of the Pfaffian varieties as follows.

\smallskip

First note that the log resolution of $\Pf(2k,V^{\vee})$ we used in this paper naturally induces a log resolution of $H$, by pulling back $H$ along the product $\widehat{\Pf(2k,V^{\vee})}\times \PP W\rightarrow \Pf(2k,V^{\vee})\times \PP W$, or equivalently, the fiber product $H\times_{\Pf(2k,V^{\vee})}\widehat{\Pf(2k,V^{\vee})}$. The exceptional divisors of this log resolution are the pullbacks of the exceptional divisors of $\Pf(2k,V^{\vee})$, and the modified discrepancies are defined to be the same as those of $\Pf(2k,V^{\vee})$. More precisely, let $F_j$ be the pullback of the exceptional divisor $D_j$ of $\widehat{\Pf(2k,V^{\vee})}$, and denote $F_J^{\circ}=(\bigcap_{j\in J}F_j)\backslash(\bigcup_{j\not\in J}F_j)$. Then the modified stringy $E$-function of $H$ is defined to be
\begin{align}\label{eq-modified stringy E-function of H}
	\widetilde{E}_{\mathrm{st}}(H) = \sum_{J} E(F_J^{\circ})\prod_{j\in J}\frac{q-1}{q^{\widetilde{\alpha}_{j,k,n}+1}-1}.
\end{align}

\medskip

\begin{lemma}\label{lem-projection1}
	The fiber of $\pi_1$ over a point $w\in \Pf(2k,V^{\vee})$ is $\PP W$ if $w\in X_W$, and is a hyperplane $\PP^{l-2}$ in $\PP W$ if $w\not\in X_W$. Moreover, we have the equality of modified stringy $E$-functions
	\begin{align*}
		\widetilde{E}_{\mathrm{st}}(H)=\frac{q^{l-1}-1}{q-1}\frac{q^{(n-1)k}-1}{q-1}\prod_{j=k+1}^{\frac{n}{2}}\frac{q^{2j}-1}{q^{2j-2k}-1}+q^{l-1}\widetilde{E}_{\mathrm{st}}(X_W).
	\end{align*}
\end{lemma}
\begin{proof}
	There are two possibilities for the fiber of $\pi_1$ over a point $w\in \Pf(2k,V^{\vee})$. If $w\in X_W$, then $w$ is annihilated by all forms in $W$, so the fiber $\pi_1^{-1}(w)$ is the whole space $\PP W=\PP^{l-1}$. If $w\not\in X_W$, then the condition $\langle w,\alpha\rangle=0$ cuts out a hyperplane in $\PP W$, which is a projective space $\PP^{l-2}$. Using a similar argument as in the proof of Proposition \ref{prop-zariski locally trivial}, one can prove that the fibrations are Zariski locally trivial.

	\smallskip

	Now by the definition \eqref{eq-modified stringy E-function of H} and the additivity and multiplicativity of the ordinary $E$-function, we have
	\begin{equation*}
		\begin{split}
			\widetilde{E}_{\mathrm{st}}(H)&= \sum_{J}  E(D_J^{\circ}\cap \pi^{-1}(\Pf(2k,V^{\vee})\backslash X_W))E(\PP^{l-2})\prod_{j\in J}\frac{q-1}{q^{\widetilde{\alpha}_{j,k,n}+1}-1} \\
			&\quad\quad + \sum_{J}E(D_J^{\circ}\cap \pi^{-1}( X_W))E(\PP^{l-1}) \prod_{j\in J}\frac{q-1}{q^{\widetilde{\alpha}_{j,k,n}+1}-1} \\
			&=(\widetilde{E}_{\mathrm{st}}(\Pf(2k,V^{\vee}))-\widetilde{E}_{\mathrm{st}}(X_W))\cdot E(\PP^{l-2})+\widetilde{E}_{\mathrm{st}}(X_W)\cdot E(\PP^{l-1})\\
		&=\frac{q^{l-1}-1}{q-1}\frac{q^{(n-1)k}-1}{q-1}\prod_{j=k+1}^{\frac{n}{2}}\frac{q^{2j}-1}{q^{2j-2k}-1}+q^{l-1}\widetilde{E}_{\mathrm{st}}(X_W).
		\end{split}
	\end{equation*}
\end{proof}

\begin{lemma}\label{lem-projection2}
	Let $Y_i$ be the locally closed subvariety of $\PP W$ consisting of skew forms of rank $2i$. Then we have the equality of modified stringy $E$-functions
	\begin{align*}
		\widetilde{E}_{\mathrm{st}}(H)=\frac{q^{(n-1)k-1}-1}{q-1}\frac{q^{l}-1}{q-1}\prod_{j=k+1}^{\frac{n}{2}}\frac{q^{2j}-1}{q^{2j-2k}-1}+q^{(n-1)k-1}\widetilde{E}_{\mathrm{st}}(Y_W).
	\end{align*}
\end{lemma}
\begin{proof}
	We look at the fiber over a point $\alpha_i\in Y_i$ of rank $2i$. By definition this is the space of all skew forms in $\Pf(2k,V^{\vee})$ that are orthogonal to $\alpha_i$, i.e., the intersection of $\Pf(2k,V^{\vee})$ with the hyperplane $P_i$ defined by $\langle -,\alpha_i\rangle =0$. 

	Consider the projection $F_J^{\circ}\cap \pi_2^{-1}(Y_i) \rightarrow Y_i$. It is easy to see that the fibers are all isomorphic to $D_J^{\circ}\cap \pi^{-1}(\Pf(2k,V^{\vee})\cap P_i)$. Again, an argument similar to that in the proof of Proposition \ref{prop-zariski locally trivial} shows that the fibration is Zariski locally trivial. By additivity and multiplicativity of the ordinary $E$-function, we have
	\begin{align*}
		\widetilde{E}_{\mathrm{st}}(H) &= \sum_{J} E(F_J^{\circ})\prod_{j\in J}\frac{q-1}{q^{\widetilde{\alpha}_{j,k,n}+1}-1} \\
		&= \sum_{J} (\sum_{i=1}^{\frac{n}{2}} E(D_J^{\circ}\cap \pi^{-1}(\Pf(2k,V^{\vee})\cap P_i)) E(Y_i))\prod_{j\in J}\frac{q-1}{q^{\widetilde{\alpha}_{j,k,n}+1}-1} \\
		& = \sum_{i=1}^{\frac{n}{2}} (\sum_{J} E(D_J^{\circ}\cap \pi^{-1}(\Pf(2k,V^{\vee})\cap P_i))\prod_{j\in J}\frac{q-1}{q^{\widetilde{\alpha}_{j,k,n}+1}-1}) E(Y_i) \\
		& = \sum_{i=1}^{\frac{n}{2}}\widetilde{E}_{\mathrm{st}}(\Pf(2k,V^{\vee})\cap P_i)\cdot E(Y_i),
	\end{align*}
	where $\widetilde{E}_{\mathrm{st}}(\Pf(2k,V^{\vee})\cap P_i)$ is defined as
	\begin{align*}
		\sum_{J} E(D_J^{\circ}\cap \pi^{-1}(\Pf(2k,V^{\vee})\cap P_i))\prod_{j\in J}\frac{q-1}{q^{\widetilde{\alpha}_{j,k,n}+1}-1}.
	\end{align*}

	Now by Proposition \ref{stringy E function of the cut}, $\widetilde{E}_{\mathrm{st}}(\Pf(2k,V^{\vee})\cap P_i)$ is given by
	\begin{align*}
		\frac{q^{(n-1)k-1}-1}{q-1}\prod_{j=k+1}^{\frac{n}{2}}\frac{q^{2j}-1}{q^{2j-2k}-1}+q^{(n-1)k-1}\prod_{j=\frac{n}{2}-k-i+1}^{\frac{n-2i}{2}}\frac{q^{2j}-1}{q^{2j-n+2k+2i}-1}.
	\end{align*}
	Therefore
	\begin{align*}
		\widetilde{E}_{\mathrm{st}}(H)=&\frac{q^{(n-1)k-1}-1}{q-1}\prod_{j=k+1}^{\frac{n}{2}}\frac{q^{2j}-1}{q^{2j-2k}-1}\cdot\sum_{i=1}^{\frac{n}{2}}E(Y_i)\\
		&+q^{(n-1)k-1}\sum_{i=1}^{\frac{n}{2}} E(Y_i)\cdot \prod_{j=\frac{n}{2}-k-i+1}^{\frac{n-2i}{2}}\frac{q^{2j}-1}{q^{2j-n+2k+2i}-1}.
	\end{align*}
	Finally, note that $\sum_{i=1}^{\frac{n}{2}}E(Y_i)$ is exactly $E(\PP W)$. Furthermore, terms corresponding to $i>\frac{n-2k}{2}$ in the second sum vanish; hence, the second sum gives the modified stringy $E$-function of $Y_W$. The desired equality follows.
\end{proof}

\medskip

\begin{proof}[Proof of Theorem \ref{thm-comparison}]
	Comparing the equalities in Lemmas \ref{lem-projection1} and \ref{lem-projection2} yields the desired equation.
\end{proof}

\begin{remark}
	A direct computation shows that if we switch the roles of $X_W$ and $Y_W$ by making the changes $X_W\rightarrow Y_W$, $k\rightarrow \frac{n}{2}-k$, and $l\rightarrow \frac{n(n-1)}{2}-l$, the equation remains unchanged.
\end{remark}

\begin{example}
	(1) When $\dim W=nk$, we have the following relation between the modified stringy $E$-functions of $X_W$ and $Y_W$:
   \begin{align*}
   	\widetilde{E}_{\mathrm{st}}(Y_W)=q^{k}\widetilde{E}_{\mathrm{st}}(X_W)+\frac{q^{k}-1}{q-1}\prod_{j=k+1}^{\frac{n}{2}}\frac{q^{2j}-1}{q^{2j-2k}-1}
   \end{align*}
   In particular, when $n=6$ and $k=1$, the complete intersection $X_W=G(2,6)\cap\PP^8$ is a K3 surface, and $Y_W=\Pf(4,6)\cap\PP^5$ is a Pfaffian cubic fourfold. The equality 
\begin{align*}
   	E(Y_W)=q\cdot E(X_W)+1+q^2+q^4
\end{align*}
recovers the well-known observation that the Hodge diamond of a K3 surface can be embedded into the Hodge diamond of a cubic fourfold (see Figure \ref{figure-hodge diamond}).

(2) When $\dim W=(n-1)k$, the linear sections $X_W$ and $Y_W$ are of equal dimensions, and both are Fano. In this case we have the equality of modified stringy $E$-functions $\widetilde{E}_{\mathrm{st}}(X_W)=\widetilde{E}_{\mathrm{st}}(Y_W)$.
\end{example}

\begin{figure}
		\begin{tabular}{ccccccccc}
 & & &  & $1$ &  & & & \\
 & & & $0$ &  & $0$ & & & \\
&  & 0 &  & 1 &  & 0 &  & \\
& 0  &  & 0 &  & 0 & & 0 & \\
0 &  & 1 &  & 21 &  & 1 & & 0 \\
& 0  &  & 0 &  & 0 & & 0 & \\
&  & 0 &  & 1 &  & 0 &  & \\
 & & & $0$ &  & $0$ & & & \\
  & & &  & $1$ &  & & & \\
\end{tabular}=
\begin{tabular}{ccccc}
 &  & $1$ &  &  \\
 & $0$ &  & $0$ &  \\
$1$ &  & $20$ &  & $1$ \\
 & $0$ &  & $0$ &  \\
 &  & $1$ &  &  \\
\end{tabular}
+
\begin{tabular}{ccccccccc}
 & & &  & $1$ &  & & & \\
 & & &  &  &  & & & \\
&  & &  & 0 &  &  &  & \\
&   &  &  &  &  & &  & \\
 &  &  &  & 1 &  &  & &  \\
&   &  &  &  &  & &  & \\
&  &  &  & 0 &  &  &  & \\
 & & &  &  &  & & & \\
  & & &  & 1 &  & & & \\
\end{tabular}
\caption{Decomposition of the Hodge diamond of a cubic fourfold}\label{figure-hodge diamond}
\end{figure}

Using the same method, we obtain an analogous relation in the odd-dimensional cases. The main result of \cite{BL} is a special case of the following theorem when $l=nk$.

\begin{theorem}\label{thm-equality odd dim}
	Let $V$ be a complex vector space of odd dimension $n$ and $W\subseteq\wedge^2 V^{\vee}$ be a generic subspace of dimension $l$. Then we have the following relation between the stringy $E$-functions of $X_W$ and $Y_W$:
	\begin{align*}
	q^{nk}E_{\mathrm{st}}(Y_W)-q^{l}E_{\mathrm{st}}(X_W)=\frac{q^{l}-q^{nk}}{q-1}\binom{(n-1)/2}{k}_{q^2}
\end{align*}
where $\binom{(n-1)/2}{k}_{q^2}=\prod_{j=k+1}^{\frac{n-1}{2}}\frac{q^{2j}-1}{q^{2j-2k}-1}$ is the $q$-binomial coefficient.
\end{theorem}

\section{Relations to Homological Projective Duality for Pfaffian varieties}\label{sec-comments}

In this section, we make connections between our main results in this paper and the Homological Projective Duality for Pfaffian varieties. We propose a conjectural description for the shape of the Lefschetz decompositions of the $\Pf(2k,V^{\vee})$ and its dual $\Pf(n-2k,V)$, and show its compatibility with Theorem \ref{thm-comparison}.

\subsection{Generalities on Homological Projective Duality}

We start with a brief review of general definitions and facts about Homological Projective Duality. Let $X$ be a smooth projective variety, and $f:X\rightarrow \PP(V)$ be a morphism where $V$ is an $N$-dimensional vector space. A Lefschetz decomposition of $D^b(X)$ with respect to $\mathcal{O}_X(1)=f^*\mathcal{O}_{\PP(V)}(1)$ is a semiorthogonal decomposition of the form
	\begin{align*}
		D^b(X)=\langle \mathcal{A}_0 ,\mathcal{A}_1(1),\cdots, \mathcal{A}_{i-1}(i-1) \rangle
	\end{align*}
	where $\mathcal{A}_0\supseteq \mathcal{A}_1\supseteq\cdots\supseteq \mathcal{A}_{i-1}$ are admissible subcategories of $D^b(X)$. We say that such a decomposition is \textit{rectangular} if all blocks $\mathcal{A}_i$ are equal. Moreover, we say that a block $\mathcal{A}$ in the Lefschetz decomposition \textit{has size $r$} if its Grothendieck group is a free abelian group of rank $r$. In many examples these blocks admit exceptional collections of length equal to their size.
	
	\smallskip

	 We define the universal hyperplane section $\mathcal{X}\subseteq X\times\PP(V^{\vee})$ to be the subvariety consisting of pairs $(x,u)$ with $\langle x,u\rangle =0$, where $\langle-,-\rangle$ is the natural pairing between $V$ and $V^{\vee}$.
	
	\smallskip

A smooth projective variety $Y$ together with a morphism $g:Y\rightarrow \PP(V^{\vee})$ is called homologically projectively dual of $f:X\rightarrow \PP(V)$ if there is a fully faithful Fourier-Mukai transform $\Phi:D^b(Y)\hookrightarrow D^b(\mathcal{X})$ such that there is a semiorthogonal decomposition
\begin{align*}
	D^b(\mathcal{X})=\langle \Phi(D^b(Y)), \mathcal{A}_1(1)\boxtimes D^b(\PP(V^{\vee})),\cdots, \mathcal{A}_{i-1}(i-1)\boxtimes D^b(\PP(V^{\vee}))\rangle.
\end{align*} 

In this case the homological projective dual $g:Y\rightarrow\PP(V^{\vee})$ admits a dual Lefschetz decomposition 
\begin{align*}
	D^b(Y)=\langle \mathcal{B}_{j-1}(-j+1),\cdots,\mathcal{B}_{1}(-1), \mathcal{B}_{0} \rangle.
\end{align*}

The main theorem of Homological Projective Duality is a remarkable relation between the derived categories of the linear sections of the dual varieties $X$ and $Y$. More precisely, for any linear subspace $W\subseteq V^{\vee}$ with $\dim W=l$, if both $X_W:=X\cap\PP(W^{\perp})$ and $Y_W:=Y\cap\PP(W)$ are of the expected dimensions, then there are semiorthogonal decompositions
	\begin{align*}
		D^b(X_W)=\langle \mathcal{C}_W, \mathcal{A}_l(1) ,\cdots, \mathcal{A}_{i-1}(i-l) \rangle
	\end{align*}
	and
	\begin{align*}
	D^b(Y_W)=\langle \mathcal{B}_{j-1}(N-l-j),\cdots,\mathcal{B}_{N-l}(-1),\mathcal{C}_W \rangle.
\end{align*}
Apart from the blocks inherited from the ambient varieties $X$ and $Y$, there is a component $\mathcal{C}_W$ (the ``interesting part") in both of the decompositions that plays a similar role to the primitive cohomology in Hodge theory.

\smallskip

In general, the varieties $X$ and $Y$ are not necessarily smooth. In the case of a singular variety, one needs to replace the usual derived category by some kind of categorical resolution of the singularities.

\smallskip

As we have seen, the Pfaffian varieties $\Pf(2k,V)$ are highly singular, hence we have to replace them by certain categorical crepant resolutions\footnote{For a more precise definition of this, see, e.g., \cite{K2}.} of the singularities. Kuznetsov conjectured that for any $k$ and $n$, the Pfaffian varieties $\Pf(2k,V^{\vee})\hookrightarrow\PP(\wedge^2 V)$ are homologically projectively dual to $\Pf(2\floor*{\frac{n}{2}}-2k,V)\hookrightarrow\PP(\wedge^2 V^{\vee})$.

\subsection{Odd-dimensional cases}

We start by analyzing the odd-dimensional cases. A rectangular Lefschetz decomposition of the categorical crepant resolution of $\Pf(2k,V)$ in the odd-dimensional cases was constructed by Rennemo and Segal \cite{RS}:
\begin{align*}
		\widetilde{D}^b(\Pf(2k,V^{\vee}))=\Big\langle \mathcal{A}_{0}, \mathcal{A}_{1}(1),\cdots\mathcal{A}_{nk-1}(nk-1) \Big\rangle
	\end{align*}
	where $\mathcal{A}_0=\cdots=\mathcal{A}_{nk-1}$ are blocks of size $\binom{\frac{n-1}{2}}{k}$. Similarly, the dual Lefschetz decomposition of the categorical crepant resolution of $\Pf(n-1-2k,V)$ is of the form
\begin{align*}
		\widetilde{D}^b(\Pf(n-1-2k,V))=\Big\langle \mathcal{B}_{\frac{n(n-1)}{2}-nk-1}(-\frac{n(n-1)}{2}+nk+1),\cdots,\mathcal{B}_0\Big\rangle
	\end{align*}
	where $\mathcal{B}_0=\cdots=\mathcal{B}_{\frac{n(n-1)}{2}-nk-1}$ are blocks of size $\binom{\frac{n-1}{2}}{\frac{n-1-2k}{2}}=\binom{\frac{n-1}{2}}{k}$.
	
	\smallskip
	
The Homological Projective Duality statement gives semiorthogonal decompositions of the categorical resolutions of  the linear sections $X_W$ and $Y_W$ as follows:
\begin{align*}
		\widetilde{D}^b(X_W)=\langle \mathcal{C}_W, \mathcal{A}_l(1) ,\cdots, \mathcal{A}_{nk-1}(nk-l) \rangle
	\end{align*}
	and
	\begin{align*}
	\widetilde{D}^b(Y_W)=\langle \mathcal{B}_{\frac{n(n-1)}{2}-nk-1}(-nk+l),\cdots,\mathcal{B}_{\frac{n(n-1)}{2}-l}(-1),\mathcal{C}_W \rangle.
\end{align*}

Depending on the value of $l=\dim W$, the linear sections $X_W$ and $Y_W$ could be Fano, Calabi-Yau, or of general type, see Table \ref{table-odd}.

\begin{table}[htbp]
\renewcommand\arraystretch{2}
    \begin{tabular}{c|c|c|c}
 $l=\dim W$ &  $l<nk$ & $l=nk$ & $l>nk$ \\ 
\hline 
$X_W$ &  Fano & CY & general type \\ 
\hline 
$Y_W$  & general type & CY & Fano
 \\
\end{tabular}
\caption{Types of $X_W$ and $Y_W$ when $n$ is odd}\label{table-odd}
\end{table} 

When $X_W$ is Fano and $Y_W$ is of general type, the derived category on the $Y_W$ side is embedded into the one on the $X_W$ side. When $X_W$ and $Y_W$ are both Calabi-Yau, there are no components that come from the ambient Pfaffian varieties, and the two derived categories are equivalent. In particular, the ambient components never appear on both sides simultaneously in the odd-dimensional cases.

\smallskip

We consider the case when $l<nk$. In this case the derived category of $X_W$ has the decomposition
\begin{align*}
	\widetilde{D}^b(X_W)=\langle \mathcal{C}_W, \mathcal{A}_l(1) ,\cdots, \mathcal{A}_{nk-1}(nk-l) \rangle
\end{align*}
and $\widetilde{D}^b(Y_W)=\mathcal{C}_W$. The difference between the Euler characteristics on the two sides comes from the extra components on the $X_W$ side. Since all of the components are equal and have size $\binom{\frac{n-1}{2}}{k}$, the total contribution is then equal to $(nk-l)\binom{\frac{n-1}{2}}{k}$, which is compatible with the limit of the equation in Theorem \ref{thm-equality odd dim} as $q\rightarrow 1$:
\begin{align*}
	\chi(X_W)-\chi(Y_W)=(nk-l)\binom{\frac{n-1}{2}}{k}.
\end{align*}

\subsection{Even-dimensional cases}

The even-dimensional cases are more interesting: the question of establishing Homological Projective Duality remains open. This problem has been settled in the $\Pf(4,6)$ versus $G(2,6)$ case by Kuznetsov \cite{K1}. Rennemo and Segal \cite{RS} constructed a conjectural categorical crepant resolution of a general $\Pf(2k,n)$ using the technique of matrix factorizations. Pirozhkov \cite{Pirozhkov} obtained partial results in the $\Pf(n-2,n)$ versus $G(2,n)$ case by different methods.

\smallskip

An important fact is that while there exists a rectangular Lefschetz decomposition in the odd-dimensional cases, all known Lefschetz decompositions for the even-dimensional cases are non-rectangular. Indeed, all of the known examples consist of two rectangles of different sizes. Therefore, the following assumption seems reasonable.

\begin{assumption}
	The categorical crepant resolution of $\Pf(2k,n)$ has a Lefschetz decomposition consisting of two rectangles for any $k$ and even $n$.
\end{assumption}

Given this assumption, the only possible shape of the Lefschetz decomposition that is compatible with our results on the modified stringy $E$-function of $X_W$ and $Y_W$ is described in the following conjecture.

\begin{conjecture}\label{conj-lefschetz}
The categorical crepant resolution $\widetilde{D}^b(\Pf(2k,V^{\vee}))$ of the Pfaffian variety $\Pf(2k,V^{\vee})$ has a non-rectangular Lefschetz decomposition of the form
	\begin{align*}
		\widetilde{D}^b(\Pf(2k,V^{\vee}))=\Big\langle \mathcal{A}_{0}, \mathcal{A}_{1}(1),\cdots,&\mathcal{A}_{nk-\frac{n}{2}-1}(nk-\frac{n}{2}-1),\\&\mathcal{A}_{nk-\frac{n}{2}}(nk-\frac{n}{2}),\cdots,\mathcal{A}_{nk-1}(nk-1) \Big\rangle
	\end{align*}
	where $\mathcal{A}_0=\cdots=\mathcal{A}_{nk-\frac{n}{2}-1}$ are $nk-\frac{n}{2}$ blocks of size $\binom{n/2}{k}$, and $\mathcal{A}_{nk-\frac{n}{2}}=\cdots=\mathcal{A}_{nk-1}$ are $n/2$ blocks of size $\binom{n/2-1}{k}$. Similarly, the dual Lefschetz decomposition of the dual $\Pf(n-2k,V)$ has the form
	\begin{multline*}
			\widetilde{D}^b(\Pf(n-2k,V))=\Big\langle \mathcal{B}_{\frac{n^2}{2}-nk-1}(-\frac{n^2}{2}+nk+1),\\
			\cdots \mathcal{B}_{\frac{n(n-1)}{2}-nk}(-\frac{n(n-1)}{2}+nk),\\
			\mathcal{B}_{\frac{n(n-1)}{2}-nk-1}(-\frac{n(n-1)}{2}+nk+1),\cdots,\mathcal{B}_{1}(-1),\mathcal{B}_0\Big\rangle
		\end{multline*}
	where $\mathcal{B}_0=\cdots=\mathcal{B}_{\frac{n(n-1)}{2}-nk-1}$ are $\frac{n(n-1)}{2}-nk$ blocks of size $\binom{n/2}{n/2-k}$, and $\mathcal{B}_{\frac{n(n-1)}{2}-nk}=\cdots=\mathcal{B}_{\frac{n^2}{2}-nk-1}$ are $n/2$ blocks of size $\binom{n/2-1}{n/2-k}$.
\end{conjecture}

Note that this conjecture is compatible with the aforementioned results of Kuznetsov, Rennemo-Segal and Pirozhkov. According to the main theorem of Homological Projective Duality, the categorical resolutions of the linear sections $X_W=\Pf(2k,V^{\vee})\cap\PP W^{\perp}$ and $Y_W=\Pf(n-2k,V)\cap\PP W$ have semiorthogonal decompositions

\begin{align}\label{eqs-derived-x}
		\widetilde{D}^b(X_W)=\left\langle \mathcal{C}_W, \mathcal{A}_{l}(1), \cdots,\mathcal{A}_{nk-1}(nk-l) \right\rangle
\end{align}
and
\begin{align}\label{eqs-derived-y}
		\widetilde{D}^b(Y_W)=\left\langle \mathcal{B}_{\frac{n^2}{2}-nk-1}(nk-\frac{n}{2}-l), \cdots,\mathcal{B}_{\frac{n(n-1)}{2}-l}(-1), \mathcal{C}_W \right\rangle.
\end{align}

\smallskip

Depending on the value of $l=\dim W$, the linear sections $X_W$ and $Y_W$ can be Fano, Calabi-Yau or of general type, see Table \ref{table-even}.

\begin{table}[htbp]
\centering
\setlength{\tabcolsep}{4pt}
\renewcommand\arraystretch{2}
    \begin{tabular}{c|c|c|c|c|c}
 $l=\dim W$ & $l<nk-\frac{n}{2}$ & $l= nk-\frac{n}{2}$ & $nk-\frac{n}{2}<l<nk$ & $l=nk$ & $l>nk$ \\ 
\hline 
$X_W$ & Fano    
& Fano & Fano & CY & \makecell{general\\type} \\ 
\hline 
$Y_W$ & \makecell{general\\type}  & CY & Fano & Fano & Fano
 \\
\end{tabular}
\caption{Types of $X_W$ and $Y_W$ when $n$ is even}\label{table-even}
\end{table} 

In the remainder of this section, we provide evidence for this conjecture by analyzing the behavior of the stringy $E$-functions of $X_W$ and $Y_W$ as the dimension of $W$ varies.

\smallskip

An elementary calculation shows that the formula can be rewritten as
\begin{align*}
	&q^{l}\widetilde{E}_{\mathrm{st}}(X_W)-q^{(n-1)k}\widetilde{E}_{\mathrm{st}}(Y_W)\\
	&=\frac{q^{nk-\frac{n}{2}}-q^{l}}{q-1}\binom{n/2}{k}_{q^2}+q^{nk-\frac{n}{2}}\frac{q^n-1}{(q-1)(q^{\frac{n}{2}-k}+1)}\binom{n/2-1}{k}_{q^2}.
\end{align*}

\smallskip

The appearance of the non-polynomial factor on the right side remains a puzzling issue. Nevertheless, the specialization of this equation as $q\rightarrow 1$ yields a relation between the Euler characteristics of the two sides
\begin{align}\label{eqs-euler-characteristic}
	\chi(X_W)-\chi(Y_W)=(nk-l)\binom{n/2}{k}-\frac{n}{2}\binom{n/2-1}{n/2-k}.
\end{align}
On the other hand, we can get information about the difference of Euler characteristics by comparing the size of the components that appear in the decompositions \eqref{eqs-derived-x} and \eqref{eqs-derived-y}. We show that the results obtained in these two different ways coincide.

\medskip

\textbf{The case $l<nk-\frac{n}{2}$}.

\medskip

In this case $X$ is Fano and $Y_W$ is of general type. The semiorthogonal decompositions of the categorical resolutions are of the form
\begin{align*}
		\widetilde{D}^b(X_W)=\langle \mathcal{C}_W, \mathcal{A}_{l}(1),&\cdots, \mathcal{A}_{nk-\frac{n}{2}-1}(nk-\frac{n}{2}-l),\\
		 &\mathcal{A}_{nk-\frac{n}{2}}(nk-\frac{n}{2}-l+1),  \cdots,\mathcal{A}_{nk-1}(nk-l) \rangle
\end{align*}
and $\widetilde{D}^b(Y_W)=\mathcal{C_W}$. The difference between $\chi(X_W)$ and $\chi(Y_W)$ comes from the extra components 
\begin{align*}
	\mathcal{A}_{l}(1),&\cdots, \mathcal{A}_{nk-\frac{n}{2}-1}(nk-\frac{n}{2}-l),\\
		&\mathcal{A}_{nk-\frac{n}{2}}(nk-\frac{n}{2}-l+1),  \cdots,\mathcal{A}_{nk-1}(nk-l)
\end{align*}
in the decomposition of $\widetilde{D}^b(X_W)$, where the first $nk-\frac{n}{2}-l$ blocks are of size $\binom{n/2}{k}$ and the other $\frac{n}{2}$ blocks are of size $\binom{n/2-1}{k}$. Therefore the total contribution is 
\begin{align*}
	(nk-\frac{n}{2}-l)\binom{n/2}{k} + \frac{n}{2}\binom{n/2-1}{k}
\end{align*} 
which is equal to the right side of \eqref{eqs-euler-characteristic} by an elementary calculation.

\medskip

\textbf{The case $l=nk-\frac{n}{2}$}.

\medskip

In this case $X$ is Fano and $Y_W$ is Calabi-Yau. The semiorthogonal decompositions of the categorical resolutions are of the form
\begin{align*}
		\widetilde{D}^b(X_W)=\left\langle \mathcal{C}_W, \mathcal{A}_{nk-\frac{n}{2}}(1), \cdots,\mathcal{A}_{nk-1}(\frac{n}{2}) \right\rangle
\end{align*}
and $\widetilde{D}^b(Y_W)=\mathcal{C_W}$. This case is similar to the previous one, where the derived category on the $Y_W$ side is embedded into the one on the $X_W$ side. The only difference is that all the extra components are of size $\binom{n/2-1}{k}$. The difference between the Euler characteristics is given by $\frac{n}{2}\binom{n/2-1}{k}$ and is aligned with the prediction \eqref{eqs-euler-characteristic}.

\medskip

\textbf{The case $nk-\frac{n}{2}<l<nk$}.

\medskip

In this case both $X_W$ and $Y_W$ are Fano. The semiorthogonal decompositions of the categorical resolutions are of the form
\begin{align*}
		\widetilde{D}^b(X_W)=\left\langle \mathcal{C}_W, \mathcal{A}_{l}(1), \cdots,\mathcal{A}_{nk-1}(nk-l) \right\rangle
\end{align*}
and
\begin{align*}
		\widetilde{D}^b(Y_W)=\left\langle \mathcal{B}_{\frac{n^2}{2}-nk-1}(nk-\frac{n}{2}-l), \cdots,\mathcal{B}_{\frac{n(n-1)}{2}-l}(-1), \mathcal{C}_W \right\rangle
\end{align*}
This case is the most interesting one because both sides\footnote{Note that this phenomenon never happens when $n$ is odd. Presumably this is closely related to the fact that there exist rectangular Lefschetz decompositions for odd $n$ but no such decompositions are known in the even-dimensional cases.} have components coming from the ambient Pfaffian varieties: on the $X_W$ side, we have $nk-l$ blocks of size $\binom{n/2-1}{k}$, and on the $Y_W$ side, we have $l-nk+\frac{n}{2}$ blocks of size $\binom{n/2-1}{n/2-k}$. Therefore the difference between their Euler characteristics is
\begin{align*}
	(nk-l)\binom{n/2-1}{k}-(l-nk+\frac{n}{2})\binom{n/2-1}{n/2-k}
\end{align*}
which is again equal to the right side of \eqref{eqs-euler-characteristic} by an elementary calculation.

\medskip

\textbf{Other cases}.

\medskip

 The cases when $l\geq nk$ are essentially the same as the cases we have discussed above, with the only difference that the roles of $X_W$ and $Y_W$ are switched. We omit the details of the verifications.

\appendix

\section{Identities of hypergeometric functions}\label{appendix-hypergeometric}

In this appendix, we collect basic results on the $q$-hypergeometric functions that will be used in the computations in Appendix \ref{appendix-stringy E-function of cuts}. Our main reference is \cite{hypergeometric}.

\begin{definition}
	The \textit{$q$-Pochhammer symbols} and \textit{$q$-binomial symbols} are defined as
	\begin{align*}
		(a;q)_k:=\prod_{j=0}^{k-1}(1-aq^{j})
	\end{align*}
	and
	\begin{align*}
		\binom{n}{k}_{q}:=\frac{\prod_{j=0}^{k-1}(1-q^{n-j})}{\prod_{j=0}^{k-1}(1-q^{j+1})}
	\end{align*}
	for $0\leq k\leq n$, respectively.
\end{definition}

\begin{definition}
	The \textit{$q$-hypergeometric function} is defined as
	\begin{multline*}
_{r}\phi_s
\left(
\begin{array}{c}
a_1,\ldots,a_r\\
b_1,\ldots,b_s
\end{array}
;q,z
\right)\\
=\sum_{n\geq 0}
\frac {(a_1;q)_n\cdots (a_r;q)_n}
{(q;q)_n(b_1;q)_n\cdots(b_s;q)_n}
((-1)^nq^{\frac12 n(n-1)})^{1+s-r}z^n.
\end{multline*}
\end{definition}

If we write the series defining a $q$-hypergeometric function as $\sum_{s\geq 0} c_s$, then a direct computation shows that the ratio of two consecutive terms
\begin{align}\label{eqs-hypergeometric series}
	\frac{c_{n+1}}{c_n}=\frac{(1-a_1q^n)\cdots (1-a_rq^n)}{(1-q^{n+1})(1-b_1q^n)\cdots (1-b_sq^n)}(-q^n)^{1+s-r}z
\end{align}
is a rational function of $q^n$. Conversely, if a series $\sum_{s\geq 0}c_s$ satisfies such a recursive relation, then it can be written as the product of the leading term $c_0$ with a $q$-hypergeometric function. We will use this fact to reduce computations of certain stringy $E$-functions to some identities of hypergeometric functions in Appendix \ref{appendix-stringy E-function of cuts}.

\smallskip

The following identities of hypergeometric functions can be found in \cite{hypergeometric}.
\begin{proposition}
We have the following identities:
\begin{align}
	_{2}\phi_{1}
\left(
\begin{array}{c}
q^{-n},b\\
c
\end{array}
;q,\frac{cq^n}{b}
\right) &= \frac{(c/b;q)_n}{(c;q)_n} \label{eqs-hypergeometric identity 1}\\
_{2}\phi_{1}
\left(
\begin{array}{c}
q^{-n},b\\
c
\end{array}
;q,q
\right) &= \frac{(c/b;q)_n}{(c;q)_n}b^n \label{eqs-hypergeometric identity 2}\\
_{3}\phi_{2}
\left(
\begin{array}{c}
q^{-n},b,c\\
d,e
\end{array}
;q,\frac{deq^n}{bc}
\right)&=\frac{(e/c;q)_n}{(e;q)_n}\ _{3}\phi_{2}
\left(
\begin{array}{c}
q^{-n},c,d/b\\
d,cq^{1-n}/e
\end{array}
;q,q
\right)\label{eqs-hypergeometric identity 3}\\
_{4}\phi_{3}
\left(
\begin{array}{c}
q^{-n},q^{1-n},a,aq\\
qb^2,d,dq
\end{array}
;q^2,q^2
\right)&=\frac{(d/a;q)_n}{(d;q)_n}a^n\cdot\ _{4}\phi_{2}
\left(
\begin{array}{c}
q^{-n},a,b,-b\\
b^2,aq^{1-n}/d
\end{array}
;q,-\frac{q}{d}
\right)\label{eqs-hypergeometric identity 4}
\end{align}
\end{proposition}
\begin{proof}
	See \cite[(1.5.2), (1.5.3), Appendix (III.13), Exercise 3.4]{hypergeometric}.
\end{proof}

\section{Modified stringy \texorpdfstring{$E$}{E}-functions of linear sections}\label{appendix-stringy E-function of cuts}

In this appendix, we compute the modified stringy $E$-function of a hyperplane section of $\Pf(2k,V^{\vee})$. Let $\alpha_i\in\PP(\wedge^2 V^{\vee})$ be a skew form of rank $2i$, and let $P_i$ be the hyperplane defined by $\langle \alpha_i,-\rangle=0$ in the space $\PP(\wedge^2 V)$. We define the modified stringy $E$-function of the subvariety $\Pf(2k,V^{\vee})\cap P_i$ by
\begin{align*}
	\sum_{J} E(D_J^{\circ}\cap \pi^{-1}(\Pf(2k,V^{\vee})\cap P_i))\prod_{j\in J}\frac{q-1}{q^{\widetilde{\alpha}_{j,k,n}+1}-1}.
\end{align*}
The result is used in the comparison of the modified stringy $E$-functions of $X_W$ and $Y_W$ in Section \ref{sec-comparison}.

\smallskip

We first introduce some notation that will be used in the computation. We will freely use the $q$-Pochhammer symbols, $q$-binomial coefficients and $q$-hypergeometric functions defined in Appendix \ref{appendix-hypergeometric}. For simplicity, we denote the modified stringy $E$-function of the hyperplane sections above by $f_{k,i,n}$.

\begin{proposition}
	We denote the $E$-polynomial of the subvariety of $\Pf^{\circ}(2k,V^{\vee})$ cut out by the hyperplane $P_i=(\langle \alpha_i,-\rangle=0)$ in the space $\PP(\wedge^2 V)$ by $f_{k,i,n}^{\circ}$. Then we have
	\begin{align*}
		f_{k,i,n} = \sum_{1\leq p\leq k}f_{p,i,n}^{\circ}\prod_{j=k+1-p}^{\frac{n}{2}-p}\frac{q^{2j}-1}{q^{2j-2k+2p}-1}.
	\end{align*}
\end{proposition}
\begin{proof}
	We decompose the hyperplane section $\Pf(2k,V^{\vee})\cap P_i$ according to the rank of the skew forms. The stratum of forms of rank $2p$ is a fibration over the stratum $\Pf^{\circ}(2p,V^{\vee})$ of forms of rank $2p$ in $\Pf(2k,V^{\vee})$.

	\smallskip

	\begin{align*}
		& \sum_{J} E(D_J^{\circ}\cap \pi^{-1}(\Pf(2k,V^{\vee})\cap P_i))\prod_{j\in J}\frac{q-1}{q^{\widetilde{\alpha}_{j,k,n}+1}-1} \\
		= & \sum_{J}\sum_{p=1}^{\frac{n}{2}} E(D_J^{\circ}\cap \pi^{-1}(\Pf^{\circ}(2p,V^{\vee})\cap P_i))\prod_{j\in J}\frac{q-1}{q^{\widetilde{\alpha}_{j,k,n}+1}-1} 
	\end{align*}
	Denote the fiber over a general point of $\Pf^{\circ}(2p,V^{\vee})$ by
	\[
		F_p \cong \widehat{\Pf(2k-2p,n-2p)}.
	\]
	The expression above then equals
	\begin{align*}
		&\sum_{p=1}^{\frac{n}{2}} \sum_{J} E(D_J^{\circ}\cap F_p) E(\Pf^{\circ}(2p,V^{\vee})\cap P_i) \prod_{j\in J}\frac{q-1}{q^{\widetilde{\alpha}_{j,k,n}+1}-1} \\
		= & \sum_{p=1}^{\frac{n}{2}} f_{p,i,n}^{\circ} \sum_{J} E(D_J^{\circ}\cap F_p) \prod_{j\in J}\frac{q-1}{q^{\widetilde{\alpha}_{j,k,n}+1}-1}.
	\end{align*}
	Finally, note that the fiber $F_p$ is contained in the divisor $D_{\frac{n}{2}-p}$, and the intersections $D_j\cap F_p$ are exactly the exceptional divisors $E_j$ of the resolution $\widehat{\Pf(2k-2p,n-2p)}$. Therefore we have

	\begin{align*}
		\sum_{J} E(D_J^{\circ}\cap F_p) \prod_{j\in J}\frac{q-1}{q^{\widetilde{\alpha}_{j,k,n}+1}-1} &= \widetilde{E}_{\mathrm{st}}(\Pf(2k-2p,n-2p)) \frac{q-1}{q^{\widetilde{\alpha}_{\frac{n}{2}-p,k,n}}+ 1} \\
		&= \prod_{j=k+1-p}^{\frac{n}{2}-p}\frac{q^{2j}-1}{q^{2j-2k+2p}-1}.
	\end{align*}
\end{proof}

\begin{proposition}
Let $V$ be a complex vector space of dimension $n$ (not necessarily even). Let $\alpha$ be a skew form on $V$ of rank $2i$, and $k$ be a positive integer. Then the $E$-polynomial of the variety of isotropic subspaces of dimension $2k$ for the form $\alpha$ is given by
	\begin{align*}
		l_{k,i,n}=\sum_{0\leq r\leq 2k}g_{r,n-2i} q^{(2k-r)(n-2i-r)} \frac{\prod_{j=i+r+1-2k}^{i}(1-q^{2j})}{\prod_{j=1}^{2k-r}(1-q^j)}
	\end{align*}
\end{proposition}
\begin{proof}
	See \cite[Proposition A.5]{BL}.
\end{proof}

We have the following inversion formula for $f_{k,i,n}$ and $f_{p,i,n}^{\circ}$.

\begin{lemma}
The following identity holds
\begin{align*}
	f_{k,i,n}^{\circ}=\sum_{1\leq j\leq k}f_{j,i,n}(-1)^{k-j}q^{(k-j)(k-j-1)}\binom{\frac{n}{2}-j}{k-j}_{q^2}.
\end{align*}
\end{lemma}
\begin{proof}
	\begin{align*}
		\sum_{1\leq j\leq k}&f_{j,i,n}(-1)^{k-j}q^{(k-j)(k-j-1)}\binom{\frac{n}{2}-j}{k-j}_{q^2}\\
		&=\sum_{1\leq j\leq k}\left( \sum_{1\leq p\leq j}f_{p,i,n}^{\circ}\binom{\frac{n}{2}-p}{j-p}_{q^2} \right)(-1)^{k-j}q^{(k-j)(k-j-1)}\binom{\frac{n}{2}-j}{k-j}_{q^2}\\
		&=\sum_{1\leq j\leq k} \sum_{1\leq p\leq j}f_{p,i,n}^{\circ} (-1)^{k-j}q^{(k-j)(k-j-1)}\binom{\frac{n}{2}-p}{j-p}_{q^2}\binom{\frac{n}{2}-j}{k-j}_{q^2}\\
		&=\sum_{1\leq j\leq k} \sum_{0\leq s\leq k-p}f_{p,i,n}^{\circ} (-1)^{s}q^{s(s-1)}\binom{\frac{n}{2}-p}{k-s-p}_{q^2}\binom{\frac{n}{2}-k+s}{s}_{q^2}\\
		&=\sum_{1\leq j\leq k} f_{p,i,n}^{\circ} \binom{\frac{n}{2}-p}{k-p}_{q^2}\sum_{0\leq s\leq k-p} (-1)^{s}q^{s(s-1)}\binom{k-p}{s}_{q^2}\\
		&=\sum_{1\leq j\leq k} f_{p,i,n}^{\circ} \binom{\frac{n}{2}-p}{k-p}_{q^2}(1;q^2)_{k-p}\\
		&=f_{k,i,n}^{\circ}
	\end{align*}
	where in the last two equalities we used $\sum_{0\leq s\leq a} (-1)^{s}q^{s(s-1)}\binom{a}{s}_{q^2}=(1;q^2)_{a}=\delta_{a}^0$.
\end{proof}

\begin{proposition}
	For each triple $(k,i,n)$ with $n$ even and $k,i\leq\frac{n}{2}$, we have
	\begin{align*}
		\sum_{p=1}^k g_{n-2k,n-2p}f_{p,i,n}^{\circ}=\frac{q^{2k^2-k-1}-1}{q-1}g_{2k,n}+q^{2k^2-k-1}l_{k,i,n}
	\end{align*}
\end{proposition}
\begin{proof}
	This is \cite[Proposition A.7]{BL}. We give a brief review of the proof because a similar idea is used in the proof of the main result Theorem \ref{thm-comparison} of this paper.
	
	\smallskip
	
	We use the non-log resolution $\PP(\wedge^2 Q^{\vee})\rightarrow\Pf(2k,V)$ introduced in Remark \ref{rem-non-log-resolution}. We consider the space 
	\begin{align*}
		H_{\alpha}:=\left\{ (w,V_1): w\in\Pf(2k,V),\ V_1\subseteq\ker w,\ \dim V_1=n-2k,\ \langle w,\alpha\rangle =0 \right\}
	\end{align*}
	which is a hyperplane defined by $\langle w,\alpha\rangle =0$ in $\PP(\wedge^2 Q^{\vee})$. Again, we consider two projections $H_{\alpha}\rightarrow\Pf(2k,V)$ and $H_{\alpha}\rightarrow G(n-2k,V)$ and compute the $E$-function of $H_{\alpha}$ in two different ways. The fiber over a form $w\in\Pf(2k,V)$ of rank $2p$ is isomorphic to $G(n-2k,n-2p)$, which gives the LHS of the equation. On the other hand, the fiber over a point $V_1\in G(n-2k,V)$ is either isomorphic to $\PP^{2k^2-k-1}$ or $\PP^{2k^2-k-2}$, depending on whether or not $V_1^{\perp}$ is an isotropic space for $\alpha$, which gives the right side of the equation.
\end{proof}

\begin{proposition}\label{eqs used to compute f}
For each triple $(k,i,n)$ with $n$ even and $k,i\leq\frac{n}{2}$, we have
	\begin{multline*}
		\sum_{j=1}^k f_{j,i,n}q^{2(k-j)^2-(k-j)}\frac{(q^{n+2-4k+2j};q^2)_{2k-2j}}{(q;q)_{2k-2j}}\\
			=\frac{q^{2k^2-k-1}-1}{q-1}g_{2k,n}+q^{2k^2-k-1}l_{k,i,n}
	\end{multline*}
	Furthermore, these equations determine $f_{k,i,n}$ uniquely.
\end{proposition}
\begin{proof}
	First we prove a combinatorial identity
	\begin{align*}
		\sum_{s=0}^a (-1)^s q^{s^2-s}\binom{2b-2s}{2a-2s}_{q}\binom{b}{s}_{q^2}=q^{2a^2-a}\frac{(q^{2b-4a+2};q^2)_{2a}}{(q;q)_{2a}}.
	\end{align*}
	We denote the summand on the left side by $c_s$. A direct computation shows
	\begin{align*}
		\frac{c_{s+1}}{c_s}=q^{4a-2b}\cdot\frac{(1-q^{-2a+2s})(1-q^{-2a+1+2s})}{(1-q^{2s+2})(1-q^{1-2b+2s})}.
	\end{align*}
	Comparing with \eqref{eqs-hypergeometric series}, we can write the sum on the left side in terms of hypergeometric functions:
	\begin{align*}
		\binom{2b}{2a}_q\cdot\  _{2}\phi_{1}
\left(
\begin{array}{c}
q^{-2a},q^{-2a+1}\\
q^{1-2b}
\end{array}
;q^2,q^{4a-2b}
\right).
	\end{align*}
	Applying identity \eqref{eqs-hypergeometric identity 1}, we obtain
	\begin{align*}
		\binom{2b}{2a}_q\frac{(q^{2a-2b};q^2)_a}{(q^{-2b+1};q^2)_a}=q^{2a^2-a}\frac{(q^{2b-4a+2};q^2)_{2a}}{(q;q)_{2a}}
	\end{align*}
	by a direct computation.
	
	\smallskip
	
	Taking $a=k-j$ and $b=\frac{n}{2}-j$ in this identity, we obtain
	\begin{multline*}
		\sum_{s=0}^{k-j} (-1)^s q^{s^2-s}\binom{n-2j-2s}{2k-2s-2j}_{q}\binom{\frac{n}{2}-j}{s}_{q^2}\\
			=q^{2(k-j)^2-(k-j)}\frac{(q^{n+2-4k+2j};q^2)_{2k-2j}}{(q;q)_{2k-2j}}.
	\end{multline*}
	Therefore we have
	\begin{align*}
		\sum_{p=1}^k & g_{n-2k,n-2p}f_{p,i,n}^{\circ}\\
		&=\sum_{p=1}^k g_{n-2k,n-2p}\sum_{1\leq j\leq p}f_{j,i,n}(-1)^{p-j}q^{(p-j)(p-j-1)}\binom{\frac{n}{2}-j}{p-j}_{q^2}\\
		&=\sum_{1\leq j\leq k}f_{j,i,n}\left( \sum_{j\leq p\leq k}g_{n-2k,n-2p}(-1)^{p-j}q^{(p-j)(p-j-1)}\binom{\frac{n}{2}-j}{p-j}_{q^2} \right)\\
		&=\sum_{1\leq j\leq k}f_{j,i,n}\left( \sum_{j\leq p\leq k}\binom{n-2p}{n-2k}_q (-1)^{p-j}q^{(p-j)(p-j-1)}\binom{\frac{n}{2}-j}{p-j}_{q^2} \right)\\
		&=\sum_{1\leq j\leq k}f_{j,i,n}\left( \sum_{0\leq s\leq k-j}(-1)^s q^{s^2-s}\binom{n-2j-2s}{2k-2s-2j}_{q}\binom{\frac{n}{2}-j}{s}_{q^2}\right)\\
		&=\sum_{j=1}^k f_{j,i,n}q^{2(k-j)^2-(k-j)}\frac{(q^{n+2-4k+2j};q^2)_{2k-2j}}{(q;q)_{2k-2j}}
	\end{align*}
	
	Finally, note that the coefficient of $f_{k,i,n}$ on the left side is $1$. This fact allows us to solve for $f_{j,i,n}$ recursively.
\end{proof}

\begin{proposition}\label{stringy E function of the cut}
We have the following formula for $f_{k,i,n}$:
	\begin{align*}
		f_{k,i,n}&=\frac{q^{(n-1)k-1}-1}{q-1}\prod_{j=k+1}^{\frac{n}{2}}\frac{q^{2j}-1}{q^{2j-2k}-1}\\
		&\quad +q^{(n-1)k-1}\prod_{j=\frac{n}{2}-k-i+1}^{\frac{n-2i}{2}}\frac{q^{2j}-1}{q^{2j-n+2k+2i}-1}
	\end{align*}
\end{proposition}
\begin{proof}
    In terms of $q$-binomial symbols the statement of the proposition can be written as
    \begin{align}\label{eqs of f}
    	f_{k,i,n}=\frac{q^{(n-1)k-1}-1}{q-1}\binom{\frac{n}{2}}{k}_{q^2}+q^{(n-1)k-1}\binom{\frac{n-2i}{2}}{k}_{q^2}.
    \end{align}
	According to Proposition \ref{eqs used to compute f}, it suffices to prove that \eqref{eqs of f} satisfies the equations therein. We write
	\begin{align*}
		A_{k,n}=\sum_{j=0}^k \frac{q^{(n-1)j-1}-1}{q-1}\binom{\frac{n}{2}}{j}_{q^2}q^{2(k-j)^2-(k-j)}\frac{(q^{n+2-4k+2j};q^2)_{2k-2j}}{(q;q)_{2k-2j}}
	\end{align*}
	\begin{align*}
		B_{k,i,n}=\sum_{j=0}^k \binom{\frac{n-2i}{2}}{j}_{q^2} q^{2(k-j)^2-(k-j)+(n-1)j-1}\frac{(q^{n+2-4k+2j};q^2)_{2k-2j}}{(q;q)_{2k-2j}}
	\end{align*}
	\begin{align*}
		C_{k,n}=\frac{q^{2k^2-k-1}-1}{q-1}\binom{n}{2k}_{q}
	\end{align*}
	\begin{align*}
		D_{k,i,n}=q^{2k^2-k-1}\sum_{0\leq r\leq 2k}\binom{n-2i}{r}_{q} q^{(2k-r)(n-2i-r)} \frac{\prod_{j=i+r+1-2k}^{i}(1-q^{2j})}{\prod_{j=1}^{2k-r}(1-q^j)}
	\end{align*}
	It suffices to prove $A_{k,n}+B_{k,i,n}=C_{k,n}+D_{k,i,n}$ for all $(k,i,n)$. We claim
	\begin{align*}
		A_{k,n}=C_{k,n},\quad B_{k,i,n}=D_{k,i,n}.
	\end{align*}
	First we prove $A_{k,n}=C_{k,n}$. Consider
	\begin{align*}
		\sum_{j=0}^k c_{j}:=\sum_{j=0}^k \binom{\frac{n}{2}}{j}_{q^2}q^{2(k-j)^2-(k-j)}\frac{(q^{n+2-4k+2j};q^2)_{2k-2j}}{(q;q)_{2k-2j}}.
	\end{align*}
	By a direct computation we see
	\begin{align*}
		\frac{c_{j+1}}{c_j}=q^2\cdot\frac{(1-q^{2j+1-2k})(1-q^{2j-2k})}{(1-q^{2j+2})(1-q^{n+2-4k+2j})}.
	\end{align*}
	Comparing with \eqref{eqs-hypergeometric series}, we have
	\begin{align*}
		\sum_{j=0}^k c_{j}&=q^{2k^2-k}\frac{(q^{n+2-4k};q^2)_{2k}}{(q;q)_{2k}}\cdot\  _{2}\phi_{1}
\left(
\begin{array}{c}
q^{-2k},q^{-2k+1}\\
q^{n+2-4k}
\end{array}
;q^2,q^2
\right)\\
&=q^{2k^2-k}\frac{(q^{n+2-4k};q^2)_{2k}}{(q;q)_{2k}}\frac{(q^{n-2k+1};q^2)_k}{(q^{n+2-4k};q^2)_k}(q^{-2k+1})^k\\
&=\frac{(q^{n+2-4k};q^2)_{2k}}{(q;q)_{2k}}\frac{(q^{n-2k+1};q^2)_k}{(q^{n+2-4k};q^2)_k}.
	\end{align*}
	Similarly we consider
	\begin{align*}
		\sum_{j=0}^k q^{(n-1)j-1}\cdot c_{j}=\sum_{j=0}^k \binom{\frac{n}{2}}{j}_{q^2}q^{2(k-j)^2-(k-j)+(n-1)j-1}\frac{(q^{n+2-4k+2j};q^2)_{2k-2j}}{(q;q)_{2k-2j}}.
	\end{align*}
	A similar computation yields
	\begin{align*}
		&\sum_{j=0}^k q^{(n-1)j-1}\cdot c_{j}\\
&=q^{2k^2-k-1}\frac{(q^{n+2-4k};q^2)_{2k}}{(q;q)_{2k}}\cdot \  _{2}\phi_{1}
\left(
\begin{array}{c}
q^{-2k},q^{-2k+1}\\
q^{n+2-4k}
\end{array}
;q^2,q^{n+1}
\right)\\
&=q^{2k^2-k-1}\frac{(q^{n+2-4k};q^2)_{2k}}{(q;q)_{2k}}\frac{(q^{n-2k+1};q^2)_k}{(q^{n+2-4k};q^2)_k}.
	\end{align*}
	Therefore we have
	\begin{align*}
		(1-q)A_{k,n}&=\sum_{j=0}^k c_{j}-\sum_{j=0}^k q^{(n-1)j-1}\cdot c_{j}\\
		&=(1-q^{2k^2-k-1})\frac{(q^{n+2-4k};q^2)_{2k}}{(q;q)_{2k}}\frac{(q^{n-2k+1};q^2)_k}{(q^{n+2-4k};q^2)_k}\\
		&=(1-q^{2k^2-k-1})\binom{n}{2k}_q
	\end{align*}
where the last step is a direct computation. This proves $A_{k,n}=C_{k,n}$.

\smallskip

Now we prove the second equality $B_{k,i,n}=D_{k,i,n}$. Again, we write $B_{k,i,n}=\sum_{j=0}^k d_j$. The ratio of the consecutive terms is
\begin{align*}
	\frac{d_{j+1}}{d_j}&=\frac{1-q^{n-2i-2j}}{1-q^{2j+2}}\cdot q^{3+4j-4k+n-1}\frac{(1-q^{2k-2j-1})(1-q^{2k-2j})}{(1-q^{n+2-4k+2j})(1-q^{n-2j})}\\
	&=q^{n-2i+1}\frac{(1-q^{-n+2i+2j})(1-q^{-2k+1+2j})(1-q^{-2k+2j})}{(1-q^{2j+2})(1-q^{n+2-4k+2j})(1-q^{-n+2j})}.
\end{align*}
Therefore
\begin{align*}
	B_{k,i,n}&=q^{2k^2-k-1}\frac{(q^{n+2-4k};q^2)_{2k}}{(q;q)_{2k}}\\
&\qquad {}\cdot {}_{3}\phi_{2}
\left(
\begin{array}{c}
q^{-2k},q^{-n+2i},q^{-2k+1}\\
q^{-n},q^{n+2-4k}
\end{array}
;q^2,q^{n-2i+1}
\right)\\
&=q^{2k^2-k-1}\frac{(q^{n+2-4k};q^2)_{2k}}{(q;q)_{2k}}\frac{(q^{n+1-2k};q^2)_{k}}{(q^{n+2-4k};q^2)_{2k}}\\
&\qquad {}\cdot {}_{3}\phi_{2}
\left(
\begin{array}{c}
q^{-2k},q^{-2k+1},q^{-2i}\\
q^{-n},q^{1-n}
\end{array}
;q^2,q^{2}
\right)
\end{align*}
Similarly, 
\begin{align*}
	D_{k,i,n}=q^{2k^2-k-1}\binom{n-2i}{2k}_q\  _{3}\phi_{1}
\left(
\begin{array}{c}
q^{-2k},q^{-i},q^{-i}\\
q^{n+1-2i-2k}
\end{array}
;q,-q^{n+1}
\right)
\end{align*}
Applying \eqref{eqs-hypergeometric identity 4} with the parameters $(a,b,d,n,q)=(q^{-2i},q^{-i},q^{-n},2k,q)$, we obtain
\begin{align*}
	D_{k,i,n}&=q^{2k^2-k-1}\binom{n-2i}{2k}_q\frac{(q^{-n};q)_{2k}}{(q^{-n+2i};q)_{2k}}q^{4ki}\\
&\qquad {}\cdot {}_{3}\phi_{2}
\left(
\begin{array}{c}
q^{-2k},q^{1-2k},q^{-2i}\\
q^{-n},q^{1-n}
\end{array}
;q^2,q^2
\right)\\
&=q^{2k^2-k-1}\binom{n}{2k}_q\  _{3}\phi_{2}
\left(
\begin{array}{c}
q^{-2k},q^{1-2k},q^{-2i}\\
q^{-n},q^{1-n}
\end{array}
;q^2,q^2
\right)
\end{align*}
Now the desired equality follows from
\begin{align*}
	\frac{(q^{n+2-4k};q^2)_{2k}}{(q;q)_{2k}}\frac{(q^{n+1-2k};q^2)_{k}}{(q^{n+2-4k};q^2)_{2k}}=\binom{n}{2k}_q
\end{align*}
	which is a direct computation.
\end{proof}

\end{document}